\documentclass[oneside,english,reqno]{amsart}
\usepackage[T1]{fontenc}
\usepackage[latin9]{inputenc}
\usepackage{geometry}
\usepackage{babel}
\usepackage{refstyle}
\usepackage{float}
\usepackage{mathrsfs}
\usepackage{enumitem}
\usepackage{amstext}
\usepackage{amsthm}
\usepackage{amssymb}
\usepackage{stackrel}
\usepackage{graphicx}
\usepackage[all]{xy}
\PassOptionsToPackage{normalem}{ulem}
\usepackage{ulem}
\usepackage[unicode=true,pdfusetitle,
 bookmarks=true,bookmarksnumbered=false,bookmarksopen=false,
 breaklinks=false,pdfborder={0 0 1},backref=false,colorlinks=false]
 {hyperref}
\hypersetup{
 colorlinks=true,citecolor=blue,linkcolor=blue,linktocpage=true}

\makeatletter

\AtBeginDocument{\providecommand\secref[1]{\ref{sec:#1}}}
\AtBeginDocument{\providecommand\thmref[1]{\ref{thm:#1}}}
\AtBeginDocument{\providecommand\lemref[1]{\ref{lem:#1}}}
\AtBeginDocument{\providecommand\corref[1]{\ref{cor:#1}}}
\AtBeginDocument{\providecommand\exaref[1]{\ref{exa:#1}}}
\AtBeginDocument{\providecommand\figref[1]{\ref{fig:#1}}}
\AtBeginDocument{\providecommand\defref[1]{\ref{def:#1}}}
\AtBeginDocument{\providecommand\claimref[1]{\ref{claim:#1}}}
\AtBeginDocument{\providecommand\remref[1]{\ref{rem:#1}}}
\RS@ifundefined{subsecref}
  {\newref{subsec}{name = \RSsectxt}}
  {}
\RS@ifundefined{thmref}
  {\def\RSthmtxt{theorem~}\newref{thm}{name = \RSthmtxt}}
  {}
\RS@ifundefined{lemref}
  {\def\RSlemtxt{lemma~}\newref{lem}{name = \RSlemtxt}}
  {}

\numberwithin{equation}{section}
\numberwithin{figure}{section}
\numberwithin{table}{section}
\theoremstyle{plain}
\newtheorem{thm}{\protect\theoremname}[section]
  \theoremstyle{remark}
  \newtheorem{rem}[thm]{\protect\remarkname}
  \theoremstyle{definition}
  \newtheorem{defn}[thm]{\protect\definitionname}
  \theoremstyle{plain}
  \newtheorem{lem}[thm]{\protect\lemmaname}
  \theoremstyle{plain}
  \newtheorem{cor}[thm]{\protect\corollaryname}
  \theoremstyle{remark}
  \newtheorem{claim}[thm]{\protect\claimname}
  \theoremstyle{definition}
  \newtheorem{example}[thm]{\protect\examplename}
  \theoremstyle{remark}
  \newtheorem*{acknowledgement*}{\protect\acknowledgementname}

\providecommand{\MR}[1]{}

\usepackage{bbm}
\allowdisplaybreaks
\usepackage{needspace}

\setlist[enumerate]{itemsep=5pt,topsep=3pt}
\setlist[enumerate,1]{label=\textup{(}\arabic*\textup{)},ref=\arabic*}
\setlist[enumerate,2]{label=\textup{(}\alph*\textup{)},ref=\theenumi \alph*}

\newref{lem}{refcmd={Lemma \ref{#1}}}
\newref{thm}{refcmd={Theorem \ref{#1}}}
\newref{cor}{refcmd={Corollary \ref{#1}}}
\newref{sec}{refcmd={Section \ref{#1}}}
\newref{sub}{refcmd={Section \ref{#1}}}
\newref{chap}{refcmd={Chapter \ref{#1}}}
\newref{prop}{refcmd={Proposition \ref{#1}}}
\newref{exa}{refcmd={Example \ref{#1}}}
\newref{tab}{refcmd={Table \ref{#1}}}
\newref{rem}{refcmd={Remark \ref{#1}}}
\newref{def}{refcmd={Definition \ref{#1}}}
\newref{fig}{refcmd={Figure \ref{#1}}}
\newref{claim}{refcmd={Claim \ref{#1}}}

\AtBeginDocument{
  
}

\makeatother

  \providecommand{\acknowledgementname}{Acknowledgement}
  \providecommand{\claimname}{Claim}
  \providecommand{\corollaryname}{Corollary}
  \providecommand{\definitionname}{Definition}
  \providecommand{\examplename}{Example}
  \providecommand{\lemmaname}{Lemma}
  \providecommand{\remarkname}{Remark}
\providecommand{\theoremname}{Theorem}

\begin{document}

\title{Unbounded operators in Hilbert space, duality rules, characteristic
projections, and their applications}

\author{Palle Jorgensen, Erin Pearse and Feng Tian}

\address{(Palle E.T. Jorgensen) Department of Mathematics, The University
of Iowa, Iowa City, IA 52242-1419, U.S.A. }

\email{palle-jorgensen@uiowa.edu}

\urladdr{http://www.math.uiowa.edu/\textasciitilde{}jorgen/}

\address{(Erin P.J. Pearse) Department of Mathematics, California Polytechnic
State University, San Luis Obispo, CA 93407-0403, U.S.A.}

\email{epearse@calpoly.edu}

\urladdr{http://www.calpoly.edu/\textasciitilde{}epearse/}

\address{(Feng Tian) Department of Mathematics, Hampton University, Hampton,
VA 23668, U.S.A.}

\email{feng.tian@hamptonu.edu}


\begin{abstract}
Our main theorem is in the generality of the axioms of Hilbert space,
and the theory of unbounded operators. Consider two Hilbert spaces
such that their intersection contains a fixed vector space $\mathscr{D}$.
We make precise an operator theoretic linking between such two Hilbert
spaces when it is assumed that $\mathscr{D}$ is dense in one of the
two; but generally not in the other. No relative boundedness is assumed.
Nonetheless, under natural assumptions (motivated by potential theory),
we prove a theorem where a comparison between the two Hilbert spaces
is made via a specific selfadjoint semibounded operator. Applications
include physical Hamiltonians, both continuous and discrete (infinite
network models), and operator theory of reflection positivity.
\end{abstract}

\maketitle
\tableofcontents{}

\section{\label{sec:Intro}Introduction}

Quantum-mechanical observables, such as Hamiltonians, momentum operators
etc, when realized in quantized physical systems take the form of
selfadjoint operators. The case of positive measurements dictate semibounded
and selfadjoint realization. For this to work, two requirements must
be addressed: (i) choice of appropriate Hilbert space(s); and (ii)
choice of selfadjoint extension. However from the context from physics,
the candidates for observables may only be formally selfadjoint, also
called Hermitian. Hence the second question (ii). Even if the initial
Hermitian operator might have a lower bound, lower bounds for its
selfadjoint extensions is not automatic. There are choices. They dictate
the physics; and conversely. Now, there are families of selfadjoint
extensions which preserve the initial lower bound. This is the extension
theory of Friedrichs and Krein; see e.g., \cite{MR1618628,MR614775}.
Examples include free particles on an interval, particles in a number
of potential fields including delta-like potentials, the one-dimensional
Calogero problem, the Aharonov\textendash Bohm problem (see e.g.,
\cite{MR1353099,MR3282337,MR3391891,MR3474104}), and the relativistic
Coulomb problem; and precise solutions to quantization problems must
flesh out the physical selfadjoint operators their spectral resolutions. 

The setting for our main theorem (\secref{mt}) is a given pair: two
fixed Hilbert spaces, such that their intersection contains a fixed
vector space $\mathscr{D}$. In many applications, when feasible,
it is of interest to make a precise linking between such two Hilbert
spaces when it is assumed that $\mathscr{D}$ is dense in one of the
two; but generally not in the other. It is easy if the two Hilbert
spaces are given as $L^{2}\left(\mu_{i}\right)$ spaces; then the
natural means of comparison is of course via relative absolute continuity
for the two measures; and then the Radon-Nikodym derivative serves
the purpose, \secref{nd}.

Rather, the setting for our main result below is the axioms of Hilbert
space, and the theory of unbounded operators. In this generality,
we will prove theorems where a comparison between the two is made
instead with a specific selfadjoint semibounded operator, as opposed
to a Radon-Nikodym derivative. Of course the conclusions in $L^{2}$
spaces will arise as special cases.

Our motivation comes from any one of a host of diverse applications
where the initial pairs of Hilbert spaces are not given as $L^{2}$
spaces, rather they may be Dirichlet spaces, Sobolev spaces, reproducing
kernel Hilbert spaces (RKHSs), perhaps relative RKHSs; or energy-Hilbert
spaces derived from infinite networks of prescribed resistors; or
they may arise from a host of non-commutative analysis settings, e.g.,
from von Neumann algebras, Voiculescu's free probability theory \cite{MR3535474,MR3262618},
and more.

A particular, but important, special case where the comparison of
two Hilbert spaces arises is in the theory of reflection positivity
in physics. There again, the two Hilbert spaces are linked by a common
subspace, dense in the first. The setting of reflection positivity,
see e.g., \cite{MR1767902,MR887102}, lies at the crossroads of the
theory of representations of Lie groups, on the one hand, and constructive
quantum field theory on the other; here \textquotedblleft reflection
positivity\textquotedblright{} links quantum fields with associated
stochastic (Euclidean) processes. In physics, it comes from the desire
to unify quantum mechanics and relativity, two of the dominating physical
theories in the last century.

In the mathematical physics community, it is believed that Euclidean
quantum fields are easier to understand than relativistic quantum
fields. A subsequent transition from the Euclidean theory to quantum
field theory is then provided by reflection positivity, moving from
real to imaginary time, and linking operator theory on one side to
that of the other. An important tool in the correspondence between
the Euclidean side, and the side of quantum fields is a functorial
correspondence between properties of operators on one side with their
counterparts on the other. A benefit of the study of reflection is
that it allows one to take advantage of associated Gaussian measures
on suitable spaces of distributions; hence the reflection positive
Osterwalder-Schrader path spaces and associated Markov processes;
see \cite{MR887102}. Other applications to mathematical physics include
\cite{MR3076374,MR3370354,MR3395870}, and to Gaussian processes with
singular covariance density \cite{MR3231624,MR2793121}.

Our paper is organized as follows. \secref{set} spells out the setting,
and establishes notation. Let $T$ be an operator between two Hilbert
spaces. In \secref{cp}, we study the projection onto the closure
of graph$\left(T\right)$. We show among other things that, if $T$
is closed, then the corresponding block matrix has vanishing Schur-complements.
We further give a decomposition for general $T$ into a closable and
a singular part. \secref{mt} continues the study of general operators
between two Hilbert spaces; \thmref{H12} is a structure theorem which
applies to this general context. Diverse applications are given in
the remaining 4 sections, starting with Noncommutative Radon-Nikodym
derivatives in \secref{nd}, and ending with applications to discrete
analysis, graph Laplacians on infinite network-graphs with assigned
conductance functions.

\section{\label{sec:set}The setting}

In this section we recall general facts about unbounded operators,
and at the same time we introduce notation to be used later.

Our setting is a fixed separable infinite-dimensional Hilbert space.
The inner product in $\mathscr{H}$ is denoted $\left\langle \cdot,\cdot\right\rangle $,
and we are assuming that $\left\langle \cdot,\cdot\right\rangle $
is linear in the second variable. If there is more than one Hilbert
space, say $\mathscr{H}_{i}$, $i=1,2$, involved, we shall use subscript
notation in the respective inner products, so $\left\langle \cdot,\cdot\right\rangle _{i}$
is the inner product in $\mathscr{H}_{i}$. 

Let $\mathscr{H}_{1}$ and $\mathscr{H}_{2}$ be complex Hilbert spaces.
If $\mathscr{H}_{1}\xrightarrow{\;T\;}\mathscr{H}_{2}$ represents
a linear operator from $\mathscr{H}_{1}$ into $\mathscr{H}_{2}$,
we shall denote 
\begin{equation}
dom\left(T\right)=\left\{ \varphi\in\mathscr{H}_{1}\mid\mbox{\ensuremath{T\varphi} is well-defined}\right\} ,\label{eq:in1}
\end{equation}
the domain of $T$, and 
\begin{equation}
ran\left(T\right)=\left\{ T\varphi\mid\varphi\in dom\left(T\right)\right\} ,\label{eq:in2}
\end{equation}
the range of $T$. The closure of $ran\left(T\right)$ will be denoted
$\overline{ran\left(T\right)}$, and it is called the \emph{closed
range.} 
\begin{rem}
When $dom\left(T\right)$ is dense in $\mathscr{H}_{1}$ (as we standardly
assume), then we write $T:\mathscr{H}_{1}\rightarrow\mathscr{H}_{2}$
or $\mathscr{H}_{1}\xrightarrow{\;T\;}\mathscr{H}_{2}$ with the tacit
understanding that $T$ is only defined for $\varphi\in dom\left(T\right)$. 
\end{rem}
\begin{defn}
Let $T:\mathscr{H}_{1}\rightarrow\mathscr{H}_{2}$ be a densely defined
operator, and consider the subspace $dom\left(T^{*}\right)\subset\mathscr{H}_{2}$
defined as follows:
\begin{align}
dom(T^{*})= & \Big\{ h_{2}\in\mathscr{H}_{2}\mid\mbox{\ensuremath{\exists C=C_{h_{2}}<\infty,} s.t. \ensuremath{\left|\left\langle h_{2},T\varphi\right\rangle _{2}\right|\leq C\left\Vert \varphi\right\Vert _{1}}}\nonumber \\
 & \quad\mbox{holds for \ensuremath{\forall\varphi\in dom\left(T\right)}}\Big\}.\label{eq:in3}
\end{align}
Then by Riesz' theorem, there is a unique $\eta\in\mathscr{H}_{1}$
for which
\begin{equation}
\left\langle \eta,\varphi\right\rangle _{1}=\left\langle h_{2},T\varphi\right\rangle _{2},\quad h_{2}\in dom(T^{*}),\;\varphi\in dom\left(T\right),\label{eq:4}
\end{equation}
and we define the adjoint operator by $T^{*}h_{2}=\eta$.

It is clear that $T^{*}$ is an operator from $\mathscr{H}_{2}$ into
$\mathscr{H}_{1}$:
\[
\xymatrix{\mathscr{H}_{1}\ar@/^{1pc}/[rr]^{T} &  & \mathscr{H}_{2}\ar@/^{1pc}/[ll]^{T^{*}}}
\]
\end{defn}
\Needspace{3\baselineskip}
\begin{defn}
The direct sum space $\mathscr{H}_{1}\oplus\mathscr{H}_{2}$ is a
Hilbert space under the natural inner product
\begin{equation}
\left\langle \begin{bmatrix}\varphi_{1}\\
\varphi_{2}
\end{bmatrix},\begin{bmatrix}\psi_{1}\\
\psi_{2}
\end{bmatrix}\right\rangle :=\left\langle \varphi_{1},\psi_{1}\right\rangle _{\mathscr{H}_{1}}+\left\langle \varphi_{2},\psi_{2}\right\rangle _{\mathscr{H}_{2}},
\end{equation}
and the \emph{graph} of $T$ is 
\begin{equation}
G_{T}:=\left\{ \begin{bmatrix}\varphi\\
T\varphi
\end{bmatrix}\mid\varphi\in dom\left(T\right)\right\} \subset\text{\ensuremath{\mathscr{H}}}_{1}\oplus\text{\ensuremath{\mathscr{H}}}_{2}.\label{eq:in6}
\end{equation}
\end{defn}
\begin{defn}
Let $T:\mathscr{H}_{1}\rightarrow\mathscr{H}_{2}$ be a linear operator. 
\begin{enumerate}
\item \label{enu:g1}$T$ is \emph{closed} iff the graph $G_{T}$ in (\ref{eq:in6})
is closed in $\text{\ensuremath{\mathscr{H}}}_{1}\oplus\text{\ensuremath{\mathscr{H}}}_{2}$.
\item \label{enu:g2} $T$ is \emph{closable} iff $\overline{\ensuremath{G_{T}}}$
is the graph of an operator. 
\item If (\ref{enu:g2}) holds, the operator corresponding to $\overline{G_{T}}$,
denoted $\overline{T}$, is called the \emph{closure}, i.e., 
\begin{equation}
\overline{G_{T}}=G_{\overline{T}}.\label{eq:in7}
\end{equation}
\end{enumerate}
\end{defn}
\begin{rem}
It follows from (\ref{eq:in6}) that $T$ is closable iff $dom(T^{*})$
is dense in $\mathscr{H}_{2}$, see \thmref{in1}. It is not hard
to construct examples of operators $\mathscr{H}_{1}\xrightarrow{\;T\;}\mathscr{H}_{2}$
with dense domain in $\mathscr{H}_{1}$ which are not closable \cite{MR0493419}.
For systematic accounts of closable operators and their applications,
see \cite{MR0052042,MR600620}.
\end{rem}
\begin{defn}
Let $V$ be the unitary operator on $\mathscr{H}\times\mathscr{H}$,
given by 
\[
V\begin{bmatrix}\varphi\\
\psi
\end{bmatrix}=\begin{bmatrix}-\psi\\
\varphi
\end{bmatrix}.
\]
Note that $V^{2}=-I$, so that any subspace is invariant under $V^{2}$. 
\end{defn}
The following two results may be found in \cite{MR0493419} or \cite{MR1157815};
see also \cite{MR2953553}.
\begin{lem}
\label{lem:gT}If $dom\left(T\right)$ is dense, then $G_{T^{*}}=\left(VG_{T}\right)^{\perp}$.
\end{lem}
\begin{proof}
Direct computation: 
\begin{align*}
\begin{bmatrix}\varphi\\
\psi
\end{bmatrix}\in G_{T^{*}} & \Longleftrightarrow\left\langle T\eta,\varphi\right\rangle =\left\langle \eta,\psi\right\rangle ,\;\forall\eta\in dom\left(T\right)\\
 & \Longleftrightarrow\left\langle \begin{bmatrix}\varphi\\
\psi
\end{bmatrix},\begin{bmatrix}-T\eta\\
\eta
\end{bmatrix}\right\rangle =0,\;\forall\eta\in dom\left(T\right)\\
 & \Longleftrightarrow\begin{bmatrix}\varphi\\
\psi
\end{bmatrix}\in\left(VG_{T}\right)^{\perp}.
\end{align*}
\end{proof}
\begin{thm}
\label{thm:in1}If $dom\left(T\right)$ is dense, then 
\begin{enumerate}
\item \label{enu:b1}$T^{*}$ is closed.
\item \label{enu:b2}$T$ is closable $\Longleftrightarrow$ $dom\left(T^{*}\right)$
is dense.
\item \label{enu:b3}$T$ is closable $\Longrightarrow$ $(\overline{T})^{*}=T^{*}$. 
\end{enumerate}
\end{thm}
\begin{proof}
(\ref{enu:b1}) This is immediate from \lemref{gT}, since $U^{\perp}$
is closed for any $U$. 

For (\ref{enu:b2}), closability gives 
\begin{alignat*}{2}
G_{\overline{T}}=\overline{G_{T}}=\left(G_{T}^{\perp}\right)^{\perp}= & \left(V^{2}G_{T}^{\perp}\right)^{\perp} & \qquad & V^{2}=I\\
= & \left(V\left(VG_{T}\right)^{\perp}\right)^{\perp} &  & V\;\text{is unitary}\\
= & \left(VG_{T^{*}}\right)^{\perp} &  & \text{part }\left(\ref{enu:b1}\right).
\end{alignat*}
If $dom(T^{*})$ is dense, then (\ref{enu:b1}) applies again to give
$G_{\overline{T}}=G_{T^{**}}$. 

For (\ref{enu:b3}), we use (\ref{enu:b1}), then (\ref{enu:b2})
again: 
\[
T^{*}=\overline{T^{*}}=\left(T^{*}\right)^{**}=\left(T^{**}\right)^{*}=\left(\overline{T}\right)^{*}.
\]
\end{proof}
\begin{defn}
An operator $T:\mathscr{H}_{1}\rightarrow\mathscr{H}_{2}$ is \emph{bounded}
iff $dom\left(T\right)=\mathscr{H}_{1}$ and there is $C<\infty$
for which $\left\Vert T\varphi\right\Vert _{2}\leq C\left\Vert \varphi\right\Vert _{1}$,
$\forall\varphi\in\mathscr{H}_{1}$. In this case, the \emph{norm}
of $T$ is 
\begin{equation}
\left\Vert T\right\Vert :=\inf\left\{ C\mid\left\Vert T\varphi\right\Vert _{2}\leq C\left\Vert \varphi\right\Vert _{1},\;\forall\in\mathscr{H}_{1}\right\} ,
\end{equation}
and it satisfies 
\begin{equation}
\left\Vert T\right\Vert =\left\Vert T^{*}\right\Vert =\left\Vert T^{*}T\right\Vert ^{1/2}.
\end{equation}
Sometimes, we clarify the notation with a subscript, e.g., $\left\Vert T\right\Vert _{\mathscr{H}_{1}\rightarrow\mathscr{H}_{2}}$
and $\left\Vert T^{*}\right\Vert _{\mathscr{H}_{2}\rightarrow\mathscr{H}_{1}}$. 
\end{defn}
\begin{thm}[von Neumann \cite{MR1502991,MR1157815}]
\label{thm:vN}Let $\mathscr{H}_{i}$, $i=1,2$, be two Hilbert spaces,
and let $T$ be a closed operator from $\mathscr{H}_{1}$ into $\mathscr{H}_{2}$
having dense domain in $\mathscr{H}_{1}$; then $T^{*}T$ is selfadjoint
in $\mathscr{H}_{1}$, $TT^{*}$ is selfadjoint in $\mathscr{H}_{2}$,
both with dense domains; and there is a \uline{partial isometry}
$J$ from $\mathscr{H}_{1}$ into $\mathscr{H}_{2}$ such that 
\begin{equation}
T=J\left(T^{*}T\right)^{\frac{1}{2}}=\left(TT^{*}\right)^{\frac{1}{2}}J\label{eq:vn1}
\end{equation}
holds on $dom\left(T\right)$. (Equation (\ref{eq:vn1}) is called
the polar decomposition of $T$.)
\end{thm}

\section{\label{sec:cp}The characteristic projection}

While a given linear operator between a pair of Hilbert spaces, say
$T$, may in general have subtle features (dictated by the particular
application at hand), the closure of graph($T$) will be a closed
subspace of the direct sum-Hilbert space, and hence the orthogonal
projection onto this subspace will be a block matrix, i.e., this projection
is a $2\times2$ matrix with bounded operator-entries. Stone suggested
the name, the characteristic projection. It will be studied below.
Our result \thmref{cp2} is new. We further show (\corref{schur})
that every closed operators $T$ has vanishing Schur-complements for
its characteristic block-matrix.

The characteristic projection was introduced and studied by Marshall
Stone in \cite{MR0052042} as a means of understanding an operator
via its graph. For more background, see \cite{MR600620,MR2953553}. 

If $\mathscr{H}_{i}$, $i=1,2,3$ are Hilbert spaces with operators
$\mathscr{H}_{1}\xrightarrow{\;A\;}\mathscr{H}_{2}\xrightarrow{\;B\;}\mathscr{H}_{3}$,
then the domain of $BA$ is 
\[
dom\left(BA\right):=\left\{ \varphi\in dom\left(A\right)\mid A\varphi\in dom\left(B\right)\right\} ,
\]
and for $x\in dom\left(BA\right)$, we have $\left(BA\right)x=B\left(Ax\right)$.
In general, $dom\left(BA\right)$ may be $\left\{ 0\right\} $, even
if $A$ and $B$ are densely defined; see \exaref{nd2}.
\begin{defn}[Characteristic projection]
\label{def:cp1}For a densely defined linear operator $\mathscr{H}_{1}\xrightarrow{\;T\;}\mathscr{H}_{2}$,
the \emph{characteristic projection} $E=E_{T}$ of $T$ is the projection
in $\mathscr{H}_{1}\oplus\mathscr{H}_{2}$ onto $\overline{G_{T}}$,
where 
\begin{equation}
E=\begin{bmatrix}E_{11} & E_{12}\\
E_{21} & E_{22}
\end{bmatrix},
\end{equation}
and the components are bounded operators
\begin{equation}
E_{ij}:\mathscr{H}_{j}\longrightarrow\mathscr{H}_{i},\quad i,j=1,2.\label{eq:cp2}
\end{equation}
\end{defn}
\begin{rem}
Since $E$ is a projection, we have $E=E^{*}=E^{2}$, where $E=E^{*}$
implies
\begin{equation}
E_{11}=E_{11}^{*}\geq0,\;E_{12}=E_{21}^{*},\;E_{21}=E_{12}^{*},\;E_{22}=E_{22}^{*}\geq0,
\end{equation}
where the ordering refers to the natural order on selfadjoint operators,
and also $E=E^{2}$ implies 
\begin{equation}
E_{ij}=E_{i1}E_{1j}+E_{i2}E_{2j},\quad i,j=1,2.
\end{equation}
\end{rem}
\begin{lem}
If $U$ is any unitary operator on $\mathscr{H}$ and $\mathscr{K}\subset\mathscr{H}$
is a subspace, then the orthogonal projection $\left(U\mathscr{K}\right)^{\perp}$
is given by 
\begin{equation}
\text{proj}\left[\left(U\mathscr{K}\right)^{\perp}\right]=I-UPU^{*},\label{eq:cp5}
\end{equation}
where $P=P_{\mathscr{K}}$ is the projection to $\mathscr{K}$. 
\end{lem}
\begin{proof}
It is obvious that (\ref{eq:cp5}) is selfadjoint and easy to check
that it is idempotent. It is also easy to check that $\left\langle \left(I-UPU^{*}\right)\varphi,U\psi\right\rangle =0$
whenever $\psi\in\mathscr{K}$. 
\end{proof}
\begin{lem}
Let $E=E_{T}$ be the characteristic projection of a closable operator
$T$. In terms of the components (\ref{eq:cp2}), the characteristic
project of $\mathscr{H}_{2}\xrightarrow{\;T^{*}\;}\mathscr{H}_{1}$
in $\mathscr{H}_{2}\oplus\mathscr{H}_{1}$is given by 
\begin{equation}
E_{T^{*}}=\begin{bmatrix}I-E_{22} & E_{21}\\
E_{12} & I-E_{11}
\end{bmatrix}.\label{eq:cp6}
\end{equation}
\end{lem}
\begin{proof}
Since $T$ is closable, we know $dom\left(T^{*}\right)$ is dense
(\thmref{in1}). Then (\ref{eq:cp6}) follows from the identity $G_{T^{*}}=\left(VG_{T}\right)^{\perp}$
of \lemref{gT}, which indicates that $E_{T^{*}}=I-VEV^{*}$. 
\end{proof}
\begin{rem}
Since the action of $T$ can be described in terms of (\ref{eq:cp2})
as the mapping 
\begin{equation}
\begin{bmatrix}E_{11}\varphi\\
E_{12}\psi
\end{bmatrix}\xrightarrow{\;T\;}\begin{bmatrix}E_{21}\varphi\\
E_{22}\psi
\end{bmatrix}\label{eq:cp7}
\end{equation}
it is clear that 
\begin{equation}
TE_{11}=E_{21}\quad\text{and}\quad TE_{12}=E_{22},\label{eq:cp8}
\end{equation}
for example, by putting $\varphi=0$ or $\psi=0$ in (\ref{eq:cp8});
cf. \figref{cp1}. Similarly, (\ref{eq:cp6}) yields
\begin{equation}
T^{*}\left(I-E_{22}\right)=E_{12}\quad\text{and}\quad T^{*}E_{21}=I-E_{11}.\label{eq:cp9}
\end{equation}
\end{rem}
\begin{figure}
\includegraphics[width=0.25\textwidth]{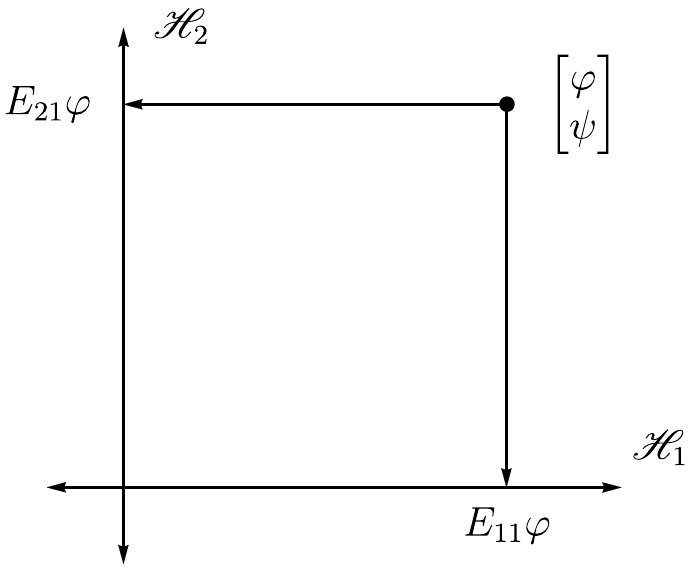}

\caption{\label{fig:cp1}A diagram indicating why $TE_{11}=E_{21}$; see (\ref{eq:cp7})
and (\ref{eq:cp8}).}

\end{figure}

\begin{thm}[{\cite[Thm. 4]{MR0052042}}]
The entries of $E=E_{T}$ are given in terms of $T$ by 
\begin{equation}
E=\begin{bmatrix}\left(I+T^{*}T\right)^{-1} & T^{*}\left(I+TT^{*}\right)^{-1}\\
T\left(I+T^{*}T\right)^{-1} & TT^{*}\left(I+TT^{*}\right)^{-1}
\end{bmatrix}.\label{eq:cp10}
\end{equation}
\end{thm}
\begin{proof}
Applying $T^{*}$ to (\ref{eq:cp8}) and then using (\ref{eq:cp9})
gives $T^{*}TE_{11}=T^{*}E_{21}=I-E_{11}$, which can be solved for
$E_{11}$ as $E_{11}=\left(I+T^{*}T\right)^{-1}$, whence another
application of $T$ (and (\ref{eq:cp8})) gives $E_{21}=T\left(I+T^{*}T\right)^{-1}$. 

Now applying $T$ to (\ref{eq:cp9}) and then using (\ref{eq:cp8})
gives $TT^{*}\left(I-E_{22}\right)=TE_{12}=E_{22}$, whence $I-E_{22}=\left(I+TT^{*}\right)^{-1}$
$\Longrightarrow$ $E_{12}=T^{*}\left(I+TT^{*}\right)^{-1}$, by (\ref{eq:cp9}),
and applying $T$ to this last one gives $E_{22}=TT^{*}\left(I+TT^{*}\right)^{-1}$. 
\end{proof}
\begin{rem}
Many more identities can be recovered from (\ref{eq:cp7}) in this
way. For example, applying $T^{*}$ to (\ref{eq:cp8}) and then using
(\ref{eq:cp9}) also gives $T^{*}TE_{12}=T^{*}E_{22}=T^{*}-E_{12}$,
which can be solved these for $E_{12}$ to give
\begin{equation}
E_{12}=\left(I+T^{*}T\right)^{-1}T^{*}.\label{eq:cp11}
\end{equation}
Now applying $T$ to (\ref{eq:cp9}) and then using (\ref{eq:cp8})
gives 
\begin{alignat*}{1}
TT^{*}E_{21}=T\left(I-E_{11}\right) & =T-E_{21},\;\text{and}\\
TT^{*}\left(I-E_{22}\right) & =TE_{12}=E_{22}.
\end{alignat*}
Solving these for $E_{22}$ and $E_{21}$, respectively, gives
\begin{equation}
E_{21}=\left(I+TT^{*}\right)^{-1}T,\quad E_{22}=\left(I+TT^{*}\right)^{-1}TT^{*}.\label{eq:cp13}
\end{equation}

On the other hand, applying (\ref{eq:cp8}) to (\ref{eq:cp11}) gives
$E_{22}=T\left(I+T^{*}T\right)^{-1}T^{*}$, and applying (\ref{eq:cp9})
to (\ref{eq:cp13}) yields 
\begin{align*}
I-E_{11} & =T^{*}T\left(I+T^{*}T\right)^{-1},\\
I-E_{22} & =I-\left(I+TT^{*}\right)^{-1}TT^{*},\\
E_{11} & =I-T^{*}T\left(1+T^{*}T\right)^{-1},\\
E_{12} & =T^{*}-T^{*}\left(I+TT^{*}\right)^{-1}TT^{*}.
\end{align*}
A summary of the above: 
\begin{align*}
E_{11} & =\left(I+T^{*}T\right)^{-1}=I-T^{*}T\left(1+T^{*}T\right)^{-1},\\
E_{12} & =\left(I+T^{*}T\right)^{-1}T^{*}=T^{*}\left(I+TT^{*}\right)^{-1}=T^{*}-T^{*}\left(I+TT^{*}\right)^{-1}TT^{*},\\
E_{21} & =\left(I+TT^{*}\right)^{-1}T=T\left(I+T^{*}T\right)^{-1},\\
E_{22} & =\left(I+TT^{*}\right)^{-1}TT^{*}=TT^{*}\left(I+TT^{*}\right)^{-1}=I-\left(I+TT^{*}\right)^{-1}TT^{*}.
\end{align*}
\end{rem}
\begin{defn}
For a matrix $X$ with block decomposition 
\[
X=\begin{bmatrix}A & B\\
C & D
\end{bmatrix},
\]
the \emph{Schur complements} (see \cite{MR2160825}) are 
\begin{equation}
X/A:=D-CA^{-1}B\quad\text{and}\quad X/D:=A-BD^{-1}C.\label{eq:cp17}
\end{equation}
\end{defn}
\begin{cor}
\label{cor:schur}A closed operator $T$ has Schur complements 
\[
E_{T}/E_{11}=E_{T}/E_{22}=0.
\]
 
\end{cor}
\begin{proof}
Computing directly from (\ref{eq:cp10}) substituted into (\ref{eq:cp17}),
we have 
\begin{align*}
E_{T}/E_{11} & =TT^{*}\left(I+TT^{*}\right)^{-1}-T\left(I+T^{*}T\right)^{-1}\left(\left(I+T^{*}T\right)^{-1}\right)^{-1}T^{*}\left(I+TT^{*}\right)^{-1}\\
 & =TT^{*}\left(I+TT^{*}\right)^{-1}-T\left(I+T^{*}T\right)^{-1}\left(I+T^{*}T\right)T^{*}\left(I+TT^{*}\right)^{-1}\\
 & =TT^{*}\left(I+TT^{*}\right)^{-1}-TT^{*}\left(I+TT^{*}\right)^{-1}=0,\quad\text{and}\\
E_{T}/E_{22} & =\left(I+T^{*}T\right)^{-1}-T^{*}\left(I+TT^{*}\right)^{-1}\left(TT^{*}\left(I+TT^{*}\right)^{-1}\right)^{-1}T\left(I+T^{*}T\right)^{-1}\\
 & =\left(I+T^{*}T\right)^{-1}-T^{*}\left(I+TT^{*}\right)^{-1}\left(I+TT^{*}\right)\left(T^{*}\right)^{-1}T^{-1}T\left(I+T^{*}T\right)^{-1}\\
 & =\left(I+T^{*}T\right)^{-1}-\left(I+T^{*}T\right)^{-1}=0.
\end{align*}
\end{proof}
\begin{lem}[{\cite[Thm. 2]{MR0052042}}]
Let $T$ be a densely defined linear operator and let $E=E_{T}$
be its characteristic projection, with components $\left(E_{ij}\right)_{i,j=1}^{2}$
as in (\ref{eq:cp2}). Then $T$ is closable if and only if $\ker\left(I-E_{22}\right)=0$,
i.e., iff
\[
\forall\psi\in\mathscr{H}_{2},\quad E_{22}\psi=\psi\Longrightarrow\psi=0.
\]
\end{lem}
\begin{proof}
Note that $E$ fixes $\overline{G_{T}}$ by definition, so $\begin{bmatrix}0\\
\psi
\end{bmatrix}\in\overline{G_{T}}$ is equivalent to 
\[
\begin{bmatrix}0\\
\psi
\end{bmatrix}=\begin{bmatrix}E_{11} & E_{12}\\
E_{21} & E_{22}
\end{bmatrix}\begin{bmatrix}0\\
\psi
\end{bmatrix}=\begin{bmatrix}E_{12}\psi\\
E_{22}\psi
\end{bmatrix}
\]
which is equivalent to $\psi\in\ker\left(E_{12}\right)\cap\ker\left(I-E_{22}\right)$.
However, from (\ref{eq:cp9}), we have 
\[
T^{*}\left(\psi-E_{22}\psi\right)=E_{12}\psi,\quad\forall\psi\in\mathscr{H}_{2},
\]
and this shows that $\ker\left(I-E_{22}\right)\subset\ker\left(E_{12}\right)$,
whereby $\begin{bmatrix}0\\
\psi
\end{bmatrix}\in\overline{G_{T}}$ iff $\psi\in\ker\left(I-E_{22}\right)$. It is clear that $T$ is
closable iff such a $\psi$ must be $0$. 
\end{proof}
\begin{thm}
\label{thm:cp2}Let $T:\mathscr{H}_{1}\rightarrow\mathscr{H}_{2}$
be a densely defined linear operator (not assumed closable) with characteristic
projection $E_{T}$ as in \defref{cp1}. Then $T$ has a maximal closable
part $T_{clo}$, defined on the domain $dom\left(T_{clo}\right):=dom\left(T\right)$,
and given by 
\begin{equation}
T_{clo}x:=\lim_{n\rightarrow\infty}\frac{1}{n+1}\sum_{k=1}^{\infty}k\,E_{22}^{n-k}E_{21}x,\quad x\in dom\left(T_{clo}\right).
\end{equation}
Let $Q$ be the projection onto $\overline{\left(I-E_{22}\right)\mathscr{H}_{2}}=\ker\left(I-E_{22}\right)^{\perp}$.
Then the characteristic projection of $T_{clo}$ is given by 
\begin{equation}
E_{T_{clo}}=\begin{bmatrix}E_{11} & E_{12}Q\\
QE_{21} & E_{22}Q
\end{bmatrix}.
\end{equation}
\end{thm}
\begin{proof}
An application of ergodic Yosida's theorem and the associated the
Cesaro mean, see e.g., \cite{MR0180824}.
\end{proof}

\section{\label{sec:mt}A duality theorem}

In this section we return to the setting where a pair of Hilbert spaces
$\mathscr{H}_{1}$ and $\mathscr{H}_{2}$ with the following property,
there is a common subspace $\mathscr{D}$ which in turn defines an
operator from $\mathscr{H}_{1}$ to $\mathscr{H}_{2}$. Its properties
are given in \thmref{H12} below.
\begin{thm}
\label{thm:H12}Let $\mathscr{H}_{i}$ be Hilbert spaces with inner
products $\left\langle \cdot,\cdot\right\rangle _{i}$, $i=1,2$.
Let $\mathscr{D}$ be a vector space s.t. $\mathscr{D}\subset\mathscr{H}_{1}\cap\mathscr{H}_{2}$,
and suppose 
\begin{equation}
\mathscr{D}\mbox{ is dense in \ensuremath{\mathscr{H}_{1}.}}\label{eq:t1}
\end{equation}
Set $\mathscr{D}^{*}\subset\mathscr{H}_{2}$, 
\begin{equation}
\mathscr{D}^{*}=\left\{ h\in\mathscr{H}_{2}\mid\exists C_{h}<\infty\mbox{ s.t. }\left|\left\langle \varphi,h\right\rangle _{2}\right|\leq C_{h}\left\Vert \varphi\right\Vert _{1},\:\forall\varphi\in\mathscr{D}\right\} ;\label{eq:t2}
\end{equation}
then the following two conditions (i)-(ii) are equivalent:

\begin{enumerate}[label=\textup{(}\roman{enumi}\textup{)},ref=\roman{enumi}]
\item \label{enu:t1}$\mathscr{D}^{*}$ is dense in $\mathscr{H}_{2}$;
and 
\item \label{enu:t2}there is a selfadjoint operator $\Delta$ with dense
domain in $\mathscr{H}_{1}$ s.t. $\mathscr{D}\subset dom\left(\Delta\right)$,
and 
\begin{equation}
\left\langle \varphi,\Delta\varphi\right\rangle _{1}=\left\Vert \varphi\right\Vert _{2}^{2},\quad\forall\varphi\in\mathscr{D}.\label{eq:t3}
\end{equation}
\end{enumerate}
\end{thm}
\begin{proof}
(\ref{enu:t1})$\Longrightarrow$(\ref{enu:t2}) Assume $\mathscr{D}^{*}$
is dense in $\mathscr{H}_{2}$; then by (\ref{eq:t2}), the inclusion
operator 
\begin{equation}
J:\mathscr{H}_{1}\longrightarrow\mathscr{H}_{2},\;J\varphi=\varphi,\;\forall\varphi\in\mathscr{D}\label{eq:t4}
\end{equation}
has $\mathscr{D}^{*}\subset dom(J^{*})$; so by (\ref{enu:t1}), $J^{*}$
has dense domain in $\mathscr{H}_{2}$, and $J$ is closable. By von
Neumann's theorem (see \thmref{vN}), $\Delta:=J^{*}\overline{J}$
is selfadjoint in $\mathscr{H}_{1}$; clearly $\mathscr{D}\subset dom\left(\Delta\right)$;
and for $\varphi\in\mathscr{D}$, 
\[
\mbox{LHS}_{\left(\ref{eq:t3}\right)}=\left\langle \varphi,J^{*}J\varphi\right\rangle _{1}=\left\langle J\varphi,J\varphi\right\rangle _{2}\underset{\text{by \ensuremath{\left(\ref{eq:t4}\right)}}}{=}\left\Vert \varphi\right\Vert _{2}^{2}=\mbox{RHS}_{\left(\ref{eq:t3}\right)}.
\]
(Note that $J^{**}=\overline{J}$.)

\begin{claim}
\label{claim:Ds}$\mathscr{D}^{*}=dom(J^{*})$, $\xymatrix{\mathscr{D}\subset\mathscr{H}_{1}\ar@/^{0.8pc}/[r]^{J} & \mathscr{H}_{2}\supset\mathscr{D}^{*}\ar@/^{0.8pc}/[l]^{J^{*}}}
$
\end{claim}
\begin{proof}
$h\in dom(J^{*})\Longleftrightarrow\exists C=C_{h}<\infty$ s.t. 
\[
|\langle\underset{\text{=\ensuremath{\varphi}}}{\underbrace{J\varphi}},h\rangle_{2}|\leq C\left\Vert \varphi\right\Vert _{1},\;\forall\varphi\in\mathscr{D}\Longleftrightarrow h\in\mathscr{D}^{*}\mbox{, by definition }\left(\ref{eq:t2}\right).
\]
Since $dom(J^{*})$ is dense, $J$ is closable, and by von Neumann's
theorem $\Delta:=J^{*}\overline{J}$ is selfadjoint in $\mathscr{H}_{1}$. 
\end{proof}
(\ref{enu:t2})$\Longrightarrow$(\ref{enu:t1}) Assume (\ref{enu:t2});
then we get a well-defined partial isometry $K:\mathscr{H}_{1}\longrightarrow\mathscr{H}_{2}$,
by 
\begin{equation}
K\Delta^{\frac{1}{2}}\varphi=\varphi,\quad\forall\varphi\in\mathscr{D}.\label{eq:t5}
\end{equation}
Indeed, (\ref{eq:t3}) reads:
\[
\Vert\Delta^{\frac{1}{2}}\varphi\Vert_{1}^{2}=\left\langle \varphi,\Delta\varphi\right\rangle _{1}=\left\Vert \varphi\right\Vert _{2}^{2},\quad\varphi\in\mathscr{D},
\]
which means that $K$ in (\ref{eq:t5}) is a \emph{partial isometry}
with $dom\left(K\right)=K^{*}K=\overline{ran(\Delta^{\frac{1}{2}})}$;
and we set $K=0$ on the complement in $\mathscr{H}_{1}$. 

Then the following inclusion holds:
\begin{equation}
\left\{ h\in\mathscr{H}_{2}\mid K^{*}h\in dom(\Delta^{\frac{1}{2}})\right\} \subseteq\mathscr{D}^{*}.\label{eq:t6}
\end{equation}
We claim that LHS in (\ref{eq:t6}) is dense in $\mathscr{H}_{2}$;
and so (\ref{enu:t1}) is satisfied. To see that (\ref{eq:t6}) holds,
suppose $K^{*}h\in dom(\Delta^{\frac{1}{2}})$; then for all $\varphi\in\mathscr{D}$,
we have
\begin{align*}
\left|\left\langle h,\varphi\right\rangle _{2}\right| & =\left|\langle h,K\Delta^{\frac{1}{2}}\varphi\rangle_{2}\right|\\
 & =\left|\langle K^{*}h,\Delta^{\frac{1}{2}}\varphi\rangle_{1}\right|\quad\left(\text{by }\left(\ref{eq:t5}\right)\right)\\
 & =\left|\langle\Delta^{\frac{1}{2}}K^{*}h,\varphi\rangle_{1}\right|\leq\Vert\Delta^{\frac{1}{2}}K^{*}h\Vert_{1}\Vert\varphi\Vert_{1},
\end{align*}
where we used Schwarz for $\left\langle \cdot,\cdot\right\rangle _{1}$
in the last step. 
\end{proof}
\begin{cor}
\label{cor:cj1}Let $\mathscr{D}\subset\mathscr{H}_{1}\cap\mathscr{H}_{2}$
be as in the statement of \thmref{H12}, and let $J:\mathscr{H}_{1}\longrightarrow\mathscr{H}_{2}$
be the associated closable operator; see (\ref{eq:t4}). Then the
complement 
\[
\mathscr{H}_{2}\ominus\mathscr{D}=\left\{ h\in\mathscr{H}_{2}\mid\left\langle \varphi,h\right\rangle _{2}=0,\;\forall\varphi\in\mathscr{D}\right\} 
\]
satisfies $\mathscr{H}_{2}\ominus\mathscr{D}=ker(J^{*})$. 
\end{cor}
\begin{proof}
Immediate from the theorem.
\end{proof}
The following result is motivated by the operator-correspondence for
the case of two Hilbert spaces $\mathscr{H}_{i}$, $i=1,2$, when
the second $\mathscr{H}_{2}$ results as a reflection-positive version
of $\mathscr{H}_{1}$; see \cite{MR1767902} for more details. 
\begin{thm}
Let $\mathscr{D}\subset\mathscr{H}_{1}\cap\mathscr{H}_{2}$ satisfying
the condition(s) in \thmref{H12}, and let $\Delta$ be the associated
selfadjoint operator from (\ref{eq:t3}). Let $U$ be a unitary operator
in $\mathscr{H}_{1}$ which maps $\mathscr{D}$ into $dom\left(\Delta\right)$,
and s.t. 
\begin{equation}
\Delta U\varphi=U^{-1}\Delta\varphi\left(=U^{*}\Delta\varphi\right)\label{eq:rp1}
\end{equation}
holds for all $\varphi\in\mathscr{D}$. 

Then there is a selfadjoint and contractive operator $\widehat{U}$
on $\mathscr{H}_{2}$ such that 
\begin{align}
\langle\widehat{U}\varphi,\psi\rangle_{2} & =\left\langle \Delta U\varphi,\psi\right\rangle _{1}\nonumber \\
 & =\left\langle U\varphi,\Delta\psi\right\rangle _{1},\quad\forall\varphi,\psi\in\mathscr{D}.\label{eq:rp2}
\end{align}
\end{thm}
\begin{proof}
\textbf{Step 1.} We first determine $\widehat{U}\varphi\in\mathscr{H}_{2}$.
We show that the following estimate holds for the term on the RHS
in (\ref{eq:rp2}): For $\varphi,\psi\in\mathscr{D}$, we have 
\begin{alignat*}{1}
\left|\left\langle \Delta U\varphi,\psi\right\rangle _{1}\right| & =\left|\left\langle U^{*}\Delta\varphi,\psi\right\rangle _{1}\right|\quad\left(\text{by \ensuremath{\left(\ref{eq:rp1}\right)}}\right)\\
 & =\left|\left\langle \Delta\varphi,U\psi\right\rangle _{1}\right|=\left|\left\langle \varphi,\Delta U\psi\right\rangle _{1}\right|=\left|\left\langle \varphi,U\psi\right\rangle _{2}\right|\leq\left\Vert U\psi\right\Vert _{2}\left\Vert \varphi\right\Vert _{2}
\end{alignat*}
since $U\psi\in dom\left(\Delta\right)$ by the assumption. Now fix
$\varphi\in\mathscr{D}$, then by Riesz, there is therefore a $h_{2}\in\mathscr{H}_{2}$
such that $\left\langle \Delta U\varphi,\psi\right\rangle _{1}=\left\langle \varphi,h_{2}\right\rangle _{2}$,
and we set $\widehat{U}\psi=h_{2}$. 

\textbf{Step 2.} Relative to the $\mathscr{H}_{2}$-inner product
$\left\langle \cdot,\cdot\right\rangle _{2}$, we have
\begin{equation}
\langle\widehat{U}\varphi,\psi\rangle_{2}=\langle\varphi,\widehat{U}\psi\rangle_{2},\quad\forall\varphi,\psi\in\mathscr{D}.\label{eq:rp3}
\end{equation}
\emph{Proof of (\ref{eq:rp3})}: 
\begin{align*}
\mbox{LHS}_{\left(\ref{eq:rp3}\right)} & =\left\langle \Delta U\varphi,\psi\right\rangle _{2}\\
 & =\left\langle U^{*}\Delta\varphi,\psi\right\rangle _{1}\quad\left(\text{by \ensuremath{\left(\ref{eq:rp1}\right)}}\right)\\
 & =\left\langle \Delta\varphi,U\psi\right\rangle _{1}=\left\langle \varphi,\Delta U\psi\right\rangle _{1}=\langle\varphi,\widehat{U}\psi\rangle_{2}=\mbox{RHS}_{\left(\ref{eq:rp3}\right)}
\end{align*}
Hence $\widehat{U}^{*}=\widehat{U}$, where $*$ here refers to $\left\langle \cdot,\cdot\right\rangle _{2}$. 

\textbf{Step 3.} $\widehat{U}$ is contractive in $\mathscr{H}_{2}$.
Let $\varphi\in\mathscr{D}$, and estimate the absolute values as
follows:
\begin{alignat*}{2}
\left|\langle\widehat{U}\varphi,\varphi\rangle_{2}\right| & =\left|\left\langle U\varphi,\Delta\varphi\right\rangle _{1}\right|\\
 & \leq\left\langle U\varphi,\Delta U\varphi\right\rangle _{1}^{\frac{1}{2}}\left\langle \varphi,\Delta\varphi\right\rangle _{1}^{\frac{1}{2}} &  & (\mbox{by Schwarz})\\
 & =\left\langle U^{2}\varphi,\Delta\varphi\right\rangle _{1}^{\frac{1}{2}}\left\langle \varphi,\Delta\varphi\right\rangle _{1}^{\frac{1}{2}}\\
 & \leq\left\langle U^{4}\varphi,\Delta\varphi\right\rangle _{1}^{\frac{1}{4}}\left\langle \varphi,\Delta\varphi\right\rangle _{1}^{\frac{1}{2}+\frac{1}{4}} &  & (\mbox{by Schwarz})\\
 & \leq\cdots &  & (\mbox{by induction})\\
 & \leq\langle U^{2^{n}}\varphi,\Delta\varphi\rangle_{1}^{\frac{1}{2^{n}}}\left\langle \varphi,\Delta\varphi\right\rangle _{1}^{\frac{1}{2}+\frac{1}{4}+\cdots+\frac{1}{2^{n}}}.
\end{alignat*}
Taking the limit $n\longrightarrow\infty$, we get $|\langle\widehat{U}\varphi,\varphi\rangle_{2}|\leq\left\Vert \varphi\right\Vert _{2}^{2}$,
since $\left\Vert \varphi\right\Vert _{2}^{2}=\left\langle \varphi,\Delta\varphi\right\rangle _{1}$
by the theorem. Since $\widehat{U}^{*}=\widehat{U}$ by Step 2, we
conclude that 
\begin{equation}
\Vert\widehat{U}\varphi\Vert_{2}\leq\left\Vert \varphi\right\Vert _{2},\quad\forall\varphi\in\mathscr{D}.\label{eq:rp4}
\end{equation}

\textbf{Step 4.} To get contractivity also on $\mathscr{H}_{2}$,
we finally extend $\widehat{U}$, defined initially only on the closure
of $\mathscr{D}$ in $\mathscr{H}_{2}$. By \corref{cj1}, we may
set $\widehat{U}=0$ on $ker(J^{*})$ in $\mathscr{H}_{2}$. 
\end{proof}
\begin{cor}
Let $\mathscr{D}\subset\mathscr{H}_{1}\cap\mathscr{H}_{2}$, and suppose
the condition(s) in \thmref{H12} are satisfied. Set $\Delta_{1}=J^{*}J$,
and $\Delta_{2}=JJ^{*}$, i.e., the two selfadjoint operators associated
to the \uline{closed} operator $J$ from \claimref{Ds}. Let $K$
be the partial isometry in (\ref{eq:t5}); then
\begin{equation}
\left\Vert \varphi\right\Vert _{2}^{2}=\left\langle K\varphi,\Delta_{2}K\varphi\right\rangle _{2},\quad\forall\varphi\in\mathscr{D}.\label{eq:m1}
\end{equation}
\end{cor}
\begin{proof}
We shall apply \thmref{vN} to the closed operator $J$. By \thmref{H12}
(\ref{enu:t2}), we have 
\begin{alignat*}{1}
\left\Vert \varphi\right\Vert _{2}^{2}=\left\langle \varphi,\Delta_{1}\varphi\right\rangle _{1} & =\left\Vert J\varphi\right\Vert _{2}^{2}\\
 & =\Vert\Delta_{2}^{\frac{1}{2}}K\varphi\Vert_{2}^{2}\quad(\mbox{by Thm. \ref{thm:vN}})\\
 & =\left\langle K\varphi,\Delta_{2}K\varphi\right\rangle _{2}
\end{alignat*}
which is the desired conclusion (\ref{eq:m1}).
\end{proof}

\section{\label{sec:nd}Noncommutative Lebesgue-Radon-Nikodym decomposition}

The following Examples illustrate that \thmref{H12} may be considered
a non-commutative Radon-Nikodym theorem. (Also see \cite{JorgensenPearse2016}.)
\begin{example}[$\mu_{2}\ll\mu_{1}$]
\label{exa:nd1}Let $\left(X,\mathscr{B}\right)$ be a $\sigma$-compact
measure space. Let $\mu_{i}$, $i=1,2$, be two regular positive measures
defined on $\left(X,\mathscr{B}\right)$. Let $\mathscr{H}_{i}:=L^{2}\left(\mu_{i}\right)$,
$i=1,2$, and set $\mathscr{D}:=C_{c}\left(X\right)$. Then the conditions
in \thmref{H12} hold if and only if $\mu_{2}\ll\mu_{1}$ (relative
absolute continuity). 

In the affirmative case, let $f=d\mu_{2}/d\mu_{1}$ be the corresponding
Radon-Nikodym derivative, and set $\Delta:=$ the operator in $L^{2}\left(\mu_{1}\right)$
of multiplication by $f\left(=d\mu_{2}/d\mu_{1}\right)$, and (\ref{eq:t3})
from the theorem reads as follows: 
\[
\left\langle \varphi,\Delta\varphi\right\rangle _{1}=\int_{X}\left|\varphi\right|^{2}f\,d\mu_{1}=\int_{X}\left|\varphi\right|^{2}d\mu_{2}=\left\Vert \varphi\right\Vert _{2}^{2},\quad\forall\varphi\in C_{c}\left(X\right).
\]
\end{example}
The link between \exaref{nd1} and the setting in \thmref{H12} (the
general case) is as follows. 
\begin{thm}
\label{thm:nd1}Assume the hypotheses of \thmref{H12}. Then, for
every $\varphi\in\mathscr{D}$, there is a Borel measure $\mu_{\varphi}$
on $[0,\infty)$ such that 
\begin{align}
\left\Vert \varphi\right\Vert _{1}^{2} & =\mu_{\varphi}\left([0,\infty)\right)\mbox{, and}\label{eq:rn1}\\
\left\Vert \varphi\right\Vert _{2}^{2} & =\int_{0}^{\infty}\lambda\,d\mu_{\varphi}\left(\lambda\right).\label{eq:rn2}
\end{align}
 
\end{thm}
\begin{proof}
By \thmref{H12}, there is a selfadjoint operator $\Delta=J^{*}J$
satisfying (\ref{eq:t3}). Let 
\[
E_{\Delta}:\mathscr{B}\left([0,\infty)\right)\longrightarrow\mbox{projections in \ensuremath{\mathscr{H}_{1}}}
\]
be the associated projection-valued measure (i.e., $\Delta=\int_{0}^{\infty}\lambda\,E_{\Delta}\left(d\lambda\right)$),
and set 
\begin{equation}
d\mu_{\varphi}\left(\lambda\right)=\left\Vert E_{\Delta}\left(d\lambda\right)\varphi\right\Vert _{1}^{2}.\label{eq:rn4}
\end{equation}
Then it follows from the Spectral Theorem that the conclusions in
(\ref{eq:rn1}) and (\ref{eq:rn2}) hold for $\mu_{\varphi}$ in (\ref{eq:rn4}).
\end{proof}
\begin{example}[$\mu_{2}\perp\mu_{1}$]
\label{exa:nd2}Let $X=\left[0,1\right]$, and consider $L^{2}\left(X,\mu\right)$
for measures $\lambda$ and $\mu$ which are mutually singular. For
concreteness, let $\lambda$ be Lebesgue measure, and let $\mu$ be
the classical singular continuous Cantor measure. Then the support
of $\mu$ is the middle-thirds Cantor set, which we denote by $K$,
so that $\mu\left(K\right)=1$ and $\lambda\left(X\backslash K\right)=1$.
The continuous functions $C\left(X\right)$ are a dense subspace of
both $L^{2}\left(X,\lambda\right)$ and $L^{2}\left(X,\mu\right)$
(see, e.g. \cite[Ch. 2]{MR924157}). Define the ``inclusion'' operator\footnote{As a map between sets, $J$ is the inclusion map $C\left(X\right)\hookrightarrow L^{2}\left(X,\mu\right)$.
However, we are considering $C\left(X\right)\subset L^{2}\left(X,\lambda\right)$
here, and so $J$ is not an inclusion map between Hilbert spaces because
the inner products are different. Perhaps ``pseudoinclusion'' would
be a better term.} $J$ to be the operator with dense domain $C\left(X\right)$ and
\begin{equation}
J:C\left(X\right)\subset L^{2}\left(X,\lambda\right)\longrightarrow L^{2}\left(X,\mu\right)\quad\text{by}\quad J\varphi=\varphi.\label{eq:nd4}
\end{equation}

We will show that $dom\left(J^{*}\right)=\left\{ 0\right\} $, so
suppose $f\in dom\left(J^{*}\right)$. Without loss of generality,
one can assume $f\geq0$ by replacing $f$ with $\left|f\right|$,
if necessary. By definition, $f\in dom\left(J^{*}\right)$ iff there
exists $g\in L^{2}\left(X,\lambda\right)$ for which
\begin{equation}
\left\langle J\varphi,f\right\rangle _{\mu}=\int_{X}\overline{\varphi}f\,d\mu=\int_{X}\overline{\varphi}g\,d\lambda=\left\langle \varphi,g\right\rangle _{\lambda},\quad\forall\varphi\in C\left(X\right).
\end{equation}
One can choose $\left(\varphi_{n}\right)_{n=1}^{\infty}\subset C\left(X\right)$
so that $\varphi_{n}\big|_{K}=1$ and $\lim_{n\rightarrow\infty}\int_{X}\varphi_{n}d\lambda=0$
by considering the appropriate piecewise linear modifications of the
constant function $\mathbbm{1}$. For example, see \figref{nd1}.
Now we have 
\begin{equation}
\left\langle \varphi_{n},J^{*}f\right\rangle _{\lambda}=\left\langle \varphi_{n},f\right\rangle _{\mu}=\left\langle 1,f\right\rangle _{\mu}=\int_{X}\left|f\right|\,d\mu,\quad\forall n\in\mathbb{N},
\end{equation}
but $\lim_{n\rightarrow\infty}\int_{X}\varphi_{n}g\,d\lambda=0$ for
any continuous $g\in L^{2}\left(X,\lambda\right)$. Thus $\int_{X}\left|f\right|\,d\mu=0$,
so that $f=0$ $\mu$-a.e. In other words, $f=0\in L^{2}\left(X,\mu\right)$
and hence $dom\left(J^{*}\right)=\left\{ 0\right\} $, which is certainly
not dense! Thus, one can interpret the adjoint of the inclusion as
multiplication by a Radon-Nikodym derivative (``$J^{*}f=f\frac{d\mu}{d\lambda}$''),
which must be trivial when the measures are mutually singular. This
comment is made more precise in \exaref{nd1} and \thmref{nd1}. As
a consequence of this extreme situation, the inclusion operator in
(\ref{eq:nd4}) is not closable. 
\end{example}
\begin{figure}
\includegraphics[width=0.8\textwidth]{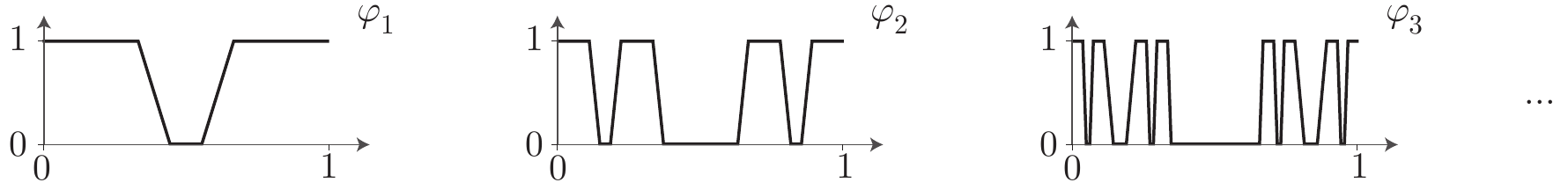}

\caption{\label{fig:nd1}A sequence $\left(\varphi_{n}\right)_{n=1}^{\infty}\subset C\left(X\right)$
for which $\varphi_{n}\big|_{K}=1$ and $\lim_{n\rightarrow\infty}\int_{X}\varphi_{n}d\lambda=0$.
See \exaref{nd2}.}

\end{figure}

\begin{rem}
Using the theory of iterated function systems (IFS), it can be shown
that for \exaref{nd2}, the inclusion in (\ref{eq:in6}) is actually
an equality, i.e., 
\[
\overline{G_{T}}=L^{2}\left(\lambda\right)\oplus L^{2}\left(\mu\right).
\]
Note that $\lambda$ and $\mu$ are both attractors of IFSs, in the
sense of Hutchinson \cite{MR625600}. Indeed, the respective IFSs
on $\left[0,1\right]$ are both given by 
\[
\left\{ S_{1}\left(x\right)=\frac{x}{r+1},\quad S_{2}\left(x\right)=\frac{x+r}{r+1}\right\} ,
\]
where $r=1$ for Lebesgue measure and $r=2$ for the Cantor measure. 
\end{rem}

\section{\label{sec:sp}The general symmetric pairs}

In this section we consider general symmetric pairs $\left(A,B\right)$,
and we show that, for every symmetric pair $\left(A,B\right)$, there
is a canonically associated single Hermitian symmetric operator $L$
in the direct sum-Hilbert space, and we show that $L$ has equal deficiency
indices. The deficiency spaces for $L$ are computed directly from
$\left(A,B\right)$.

Given $\xymatrix{\mathscr{H}_{1}\ar@/^{0.5pc}/[r]^{A} & \mathscr{H}_{2}\ar@/^{0.5pc}/[l]^{B}}
$, both linear, and assume that $dom\left(A\right)$ is dense in $\mathscr{H}_{1}$,
and $dom\left(B\right)$ is dense in $\mathscr{H}_{2}$. Assume further
that
\begin{equation}
\left\langle Au,v\right\rangle _{2}=\left\langle u,Bv\right\rangle _{1},\quad\forall u\in dom\left(A\right),\;v\in dom\left(B\right).
\end{equation}
 
\begin{thm}
\label{thm:gp}On $\mathscr{K}:=\mathscr{H}_{1}\oplus\mathscr{H}_{2}$,
set 
\begin{equation}
L\begin{bmatrix}x\\
y
\end{bmatrix}=\begin{bmatrix}By\\
Ax
\end{bmatrix},\quad\forall x\in dom\left(A\right),\:\forall y\in dom\left(B\right),\label{eq:gp2}
\end{equation}
then $L$ is symmetric (i.e., $L\subset L^{*}$) with equal deficiency
indices, i.e., 
\begin{equation}
\left\langle L\xi,\eta\right\rangle _{\mathscr{K}}=\left\langle \xi,L\eta\right\rangle _{\mathscr{K}},
\end{equation}
for all $\xi,\eta\in dom\left(L\right)=dom\left(A\right)\oplus dom\left(B\right)$. 
\end{thm}
\begin{proof}
The non-trivial part concerns the claim that $L$ in (\ref{eq:gp2})
has equal deficiency indices, i.e., the two dimensions 
\begin{equation}
\dim\left\{ \xi_{\pm}\in dom\left(L^{*}\right)\mid L^{*}\xi_{\pm}=\pm i\xi_{\pm}\right\} \label{eq:gp4}
\end{equation}
equal; we say $d_{+}=d_{-}$. 

Let $u\in\mathscr{H}_{1}$, $v\in\mathscr{H}_{2}$; then by \secref{set},
we have 
\[
\begin{bmatrix}u\\
v
\end{bmatrix}\in dom\left(L^{*}\right)\Longleftrightarrow\left[u\in dom\left(B^{*}\right),\;v\in dom\left(A^{*}\right)\right];
\]
and then 
\begin{equation}
L^{*}\begin{bmatrix}u\\
v
\end{bmatrix}=\begin{bmatrix}A^{*}v\\
B^{*}u
\end{bmatrix}.\label{eq:gp5}
\end{equation}
Now consider the following subspace in $\mathscr{K}$, 
\begin{align}
DEF:=\Big\{\begin{bmatrix}u\\
v
\end{bmatrix}\in & \mathscr{K}\mid u\in dom\left(A^{*}B^{*}\right),\:v\in dom\left(B^{*}A^{*}\right),\;\text{and}\nonumber \\
 & A^{*}B^{*}u=-u,\;B^{*}A^{*}v=-v\Big\}.\label{eq:gp6}
\end{align}
We now prove the following claim: The vectors in (\ref{eq:gp4}) both
agree with $\dim\left(\text{DEF}\right)$, see (\ref{eq:gp6}). To
see this, let $\begin{bmatrix}u\\
v
\end{bmatrix}\in\text{DEF},$ and note the following equations must then hold: 
\begin{equation}
L^{*}\begin{bmatrix}u\\
iB^{*}u
\end{bmatrix}\underset{\text{by \ensuremath{\left(\ref{eq:gp5}\right)}}}{=}\begin{bmatrix}A^{*}\left(iB^{*}u\right)\\
B^{*}u
\end{bmatrix}\underset{\text{by \ensuremath{\left(\ref{eq:gp6}\right)}}}{=}\begin{bmatrix}-iu\\
B^{*}u
\end{bmatrix}=-i\begin{bmatrix}u\\
iB^{*}u
\end{bmatrix};\label{eq:gp7}
\end{equation}
and similarly, 
\begin{equation}
L^{*}\begin{bmatrix}u\\
-iB^{*}u
\end{bmatrix}=i\begin{bmatrix}u\\
-iB^{*}u
\end{bmatrix}.\label{eq:gp8}
\end{equation}
The conclusions reverse, and we have proved that $L$ is densely defined
and symmetric with deficiency indices
\[
\left(d_{+},d_{-}\right)=\left(\dim\left(\text{DEF}\right),\dim\left(\text{DEF}\right)\right).
\]
\end{proof}

Since $L$ has equal deficiency indices we know that it has selfadjoint
extensions; see \cite{MR1502991,DS88b}. Moreover, the selfadjoint
extensions of $L$ are determined uniquely by associated partial isometries
$C$ between the respective deficiency spaces. Since we know these
deficiency spaces, see (\ref{eq:gp7}) \& (\ref{eq:gp8}), we get
the following: 
\begin{cor}
Let $A$, $B$, $\mathscr{H}_{1}$, $\mathscr{H}_{2}$, and $L$ be
as above, then TFAE:
\begin{enumerate}[label=\textup{(}\roman{enumi}\textup{)}]
\item $L$ is essentially selfadjoint,
\item $\left\{ h_{1}\in dom\left(A^{*}B^{*}\right)\mid A^{*}B^{*}h_{1}=-h_{1}\right\} =0$, 
\item $\left\{ h_{2}\in dom\left(B^{*}A^{*}\right)\mid B^{*}A^{*}h_{2}=-h_{2}\right\} =0$. 
\end{enumerate}
\end{cor}
\begin{example}[Defects $\left(d_{+},d_{-}\right)\neq\left(0,0\right)$]
Let $J$ be a finite open interval, $\mathscr{D}:=C_{c}^{2}\left(J\right)$,
i.e., compact support inside $J$, $\mathscr{H}_{1}=L^{2}\left(J\right)$,
and 
\[
\mathscr{H}_{2}:=\Big\{\text{functions \ensuremath{f} on }J/\left\{ \text{constants}\right\} \;\text{s.t.}\;\left\Vert f\right\Vert _{\mathscr{H}_{2}}^{2}:=\int_{J}\left|f'\left(x\right)\right|^{2}dx<\infty\Big\};
\]
and $\mathscr{H}_{2}$ is the Hilbert space obtained by completion
w.r.t. $\left\Vert \cdot\right\Vert _{\mathscr{H}_{2}}$.

On $\mathscr{D}\ni\varphi$, set $A\varphi:=\varphi$ mod constants;
and $Bf:=-f''=-\frac{d^{2}f}{dx^{2}}$ for $f$ such that $f''\in L^{2}$
and $f'\in L^{2}$ (the derivatives in the sense of distribution.)
Then $\left\langle A\varphi,f\right\rangle _{\mathscr{H}_{2}}=\left\langle \varphi,Bf\right\rangle _{\mathscr{H}_{1}}$
holds. So $\left(A,B,\mathscr{H}_{1},\mathscr{H}_{2}\right)$ is a
symmetric pair, and $L=\begin{bmatrix}0 & B\\
A & 0
\end{bmatrix}$ is Hermitian symmetric with dense domain in $\mathscr{K}=\left[\stackrel[\mathscr{H}_{E}]{l^{2}}{\oplus}\right]$.
One checks that the exponential function $e^{x}$ is in $dom\left(A^{*}B^{*}\right)$,
and that $A^{*}B^{*}e^{x}=-e^{x}$. 

Conclusion, the operator $L$ has deficiency indices $\left(d_{+},d_{-}\right)\neq\left(0,0\right)$.
In fact, $\left(d_{+},d_{-}\right)=\left(2,2\right)$. 
\end{example}
\begin{rem}
If the finite interval $J$ is replaced by $\left(-\infty,\infty\right)$,
then the associated operator $L=\begin{bmatrix}0 & B\\
A & 0
\end{bmatrix}$ will instead have indices $\left(d_{+},d_{-}\right)=\left(0,0\right)$. 
\end{rem}
\begin{defn}
Let $L=\begin{bmatrix}0 & B\\
A & 0
\end{bmatrix}$ be as in (\ref{eq:gp2}) acting in $\mathscr{K}=\mathscr{H}_{1}\oplus\mathscr{H}_{2}$.
The deficiency spaces $N_{i}$ and $N_{-i}$ are as follows:
\begin{align}
N_{i}\left(L^{*}\right) & =\left\{ \xi\in dom\left(L^{*}\right)\mid L^{*}\xi=i\xi\right\} \\
N_{-i}\left(L^{*}\right) & =\left\{ \eta\in dom\left(L^{*}\right)\mid L^{*}\eta=-i\eta\right\} .
\end{align}
We also set 
\begin{equation}
N_{-1}\left(A^{*}B^{*}\right)=\left\{ h\in dom\left(A^{*}B^{*}\right)\mid A^{*}B^{*}h=-h\right\} .
\end{equation}
\end{defn}
\begin{lem}
\label{lem:is}The mapping $\varphi:N_{-1}\left(A^{*}B^{*}\right)\longrightarrow N_{-i}\left(L^{*}\right)$
by 
\begin{equation}
\varphi\left(h\right)=\begin{bmatrix}h\\
iB^{*}h
\end{bmatrix},\quad\forall h\in N_{-1}\left(A^{*}B^{*}\right),\label{eq:is1}
\end{equation}
defines a linear isomorphism. 

Similarly, $\psi:N_{-1}\left(A^{*}B^{*}\right)\longrightarrow N_{i}\left(L^{*}\right)$,
by 
\begin{equation}
\psi\left(h\right)=\begin{bmatrix}h\\
-iB^{*}h
\end{bmatrix},\quad\forall h\in N_{-1}\left(A^{*}B^{*}\right)\label{eq:is2}
\end{equation}
is a linear isomorphism from $N_{-1}\left(A^{*}B^{*}\right)$ onto
$N_{i}\left(L^{*}\right)$. 

Thus the two isomorphisms are both onto: 
\[
\xymatrix{ & N_{-1}\left(A^{*}B^{*}\right)\ar[dl]_{\varphi}\ar[dr]^{\psi}\\
N_{-i}\left(L^{*}\right) &  & N_{i}\left(L^{*}\right)
}
\]
\end{lem}
\begin{proof}
Let $h\in N_{-1}\left(A^{*}B^{*}\right)$, and compute 
\[
L^{*}\begin{bmatrix}h\\
iB^{*}h
\end{bmatrix}=\begin{bmatrix}0 & A^{*}\\
B^{*} & 0
\end{bmatrix}\begin{bmatrix}h\\
iB^{*}h
\end{bmatrix}=\begin{bmatrix}-ih\\
B^{*}h
\end{bmatrix}=-i\begin{bmatrix}h\\
iB^{*}h
\end{bmatrix}.
\]
So $\varphi\left(N_{-1}\left(A^{*}B^{*}\right)\right)\subset N_{-i}\left(L^{*}\right)$.
But $\varphi$ is also onto, since 
\begin{align*}
\begin{bmatrix}h_{1}\\
h_{2}
\end{bmatrix}\in N_{-i}\left(L^{*}\right) & \Longleftrightarrow\begin{bmatrix}0 & A^{*}\\
B^{*} & 0
\end{bmatrix}\begin{bmatrix}h_{1}\\
h_{2}
\end{bmatrix}=-i\begin{bmatrix}h_{1}\\
h_{2}
\end{bmatrix}
\end{align*}
or equivalently, 
\[
\left\{ \begin{matrix}A^{*}h_{2}=-ih_{1}\\
B^{*}h_{1}=-ih_{2}
\end{matrix}\right\} .
\]
So we get $A^{*}B^{*}h_{1}=-h_{1}$, and $h_{2}=iB^{*}h_{1}$. Thus,
\[
\begin{bmatrix}h_{1}\\
h_{2}
\end{bmatrix}=\begin{bmatrix}h_{1}\\
iB^{*}h_{1}
\end{bmatrix}\in\varphi\left(N_{-1}\left(A^{*}B^{*}\right)\right)
\]
which is the claim in (\ref{eq:is1}). The proof of (\ref{eq:is2})
is similar.
\end{proof}
\begin{rem}
\label{rem:vn}By von Neumann's formulae (see \cite{DS88b}), we have
\begin{equation}
dom\left(L^{*}\right)=dom\left(L\right)+N_{i}\left(L^{*}\right)+N_{-i}\left(L^{*}\right),
\end{equation}
and there is a bijection between selfadjoint extensions $M$, i.e.,
$M\subset L\subset L^{*}$, $M=M^{*}$, and partial isometries $C:N_{i}\left(L^{*}\right)\rightarrow N_{-i}\left(L^{*}\right)$,
such that $M=L_{C}$ has 
\begin{equation}
dom\left(L_{C}\right)=\left\{ \varphi+\psi_{+}+C\psi_{+}\mid\varphi\in dom\left(L\right),\;\psi\in N_{i}\left(L^{*}\right)\right\} .\label{eq:dL}
\end{equation}
\end{rem}
\begin{rem}
Note that if $f=\varphi+\psi_{+}+\psi_{-}\in dom\left(L^{*}\right)$,
with $\varphi\in dom\left(L\right)$, $\psi_{\pm}\in N_{\pm i}\left(L^{*}\right)$,
then 
\begin{equation}
\frac{1}{2i}\left(\left\langle f,L^{*}f\right\rangle -\left\langle L^{*}f,f\right\rangle \right)=\left\Vert \psi_{+}\right\Vert ^{2}-\left\Vert \psi_{-}\right\Vert ^{2},\label{eq:bd}
\end{equation}
where the RHS of (\ref{eq:bd}) can be seen as a generalized boundary
condition. So the extensions $M$ of $L$ correspond to partial isometries
$C:N_{i}\left(L^{*}\right)\rightarrow N_{-i}\left(L^{*}\right)$. 
\end{rem}
\begin{cor}
A partial isometry 
\begin{equation}
C=\begin{bmatrix}C_{11} & C_{12}\\
C_{21} & C_{22}
\end{bmatrix}\label{eq:gp9}
\end{equation}
in $\mathscr{H}_{1}\oplus\mathscr{H}_{2}$ which determines a selfadjoint
extension of $L$ satisfies
\begin{equation}
C_{22}C_{12}^{-1}\left(C_{11}-Q\right)+C_{12}^{-1}\left(C_{11}-Q\right)Q=C_{21}\label{eq:gp10}
\end{equation}
where $Q:\ker\left(A^{*}B^{*}+I_{\mathscr{H}_{1}}\right)\longrightarrow\ker\left(A^{*}B^{*}+I_{\mathscr{H}_{1}}\right)$
is a linear automorphism. (See the diagram in \figref{cc}.)
\end{cor}
\begin{proof}
By \lemref{is}, the two deficiency spaces of $L$ are
\[
N_{\pm i}\left(L^{*}\right):=\left\{ \begin{bmatrix}u\\
\mp iB^{*}u
\end{bmatrix}\mid u\in\ker\left(A^{*}B^{*}+1\right)\right\} .
\]
Indeed, one checks that 
\[
L^{*}\begin{bmatrix}u\\
-iB^{*}u
\end{bmatrix}=\begin{bmatrix}0 & A^{*}\\
B^{*} & 0
\end{bmatrix}\begin{bmatrix}u\\
-iB^{*}u
\end{bmatrix}=i\begin{bmatrix}u\\
-iB^{*}u
\end{bmatrix},
\]
and so $\begin{bmatrix}u\\
-iB^{*}u
\end{bmatrix}\in N_{i}\left(L^{*}\right)$, with $u$ satisfying $A^{*}B^{*}u=-u$. The verification for $N_{-i}\left(L^{*}\right)$
is similar. 

By the general theory of von Neumann (see \cite{DS88b} and \remref{vn}),
the selfadjoint extensions $L_{C}\supset L$ are determined by partial
isometries $C:N_{i}\left(L^{*}\right)\rightarrow N_{-i}\left(L^{*}\right)$,
equivalently, $C$ induces a linear operator $Q:\ker\left(A^{*}B^{*}+1\right)\rightarrow\ker\left(A^{*}B^{*}+1\right)$.

Use (\ref{eq:gp6}), (\ref{eq:gp7}) and (\ref{eq:gp8}) we see that
every partial isometry $C=\left(C_{ij}\right)_{ij=1}^{2}$ as in (\ref{eq:gp9})
must satisfy 
\begin{align*}
\begin{bmatrix}C_{11} & C_{12}\\
C_{21} & C_{22}
\end{bmatrix}\begin{bmatrix}u\\
-iB^{*}u
\end{bmatrix} & =\begin{bmatrix}Qu\\
iB^{*}Qu
\end{bmatrix}\\
 & \Updownarrow\\
C_{11}u-C_{12} & iB^{*}u=Qu\\
C_{21}u-C_{22} & iB^{*}u=iB^{*}Qu
\end{align*}
It follows that $C_{12}iB^{*}=C_{11}-Q$, and $C_{22}iB^{*}+iB^{*}Q=C_{21}$.
Hence 
\[
C_{22}C_{12}^{-1}\left(C_{11}-Q\right)+C_{12}^{-1}\left(C_{11}-Q\right)Q=C_{21},
\]
which is the assertion in (\ref{eq:gp10}). 
\end{proof}
\begin{figure}[H]
\[
\xymatrix{N_{-1}\left(A^{*}B^{*}\right)\ar[rr]^{Q=\widetilde{C}}\ar[d]_{\psi} &  & N_{-1}\left(A^{*}B^{*}\right)\\
N_{i}\left(L^{*}\right)\ar[rr]_{C} &  & N_{-i}\left(L^{*}\right)\ar[u]_{\varphi^{-1}}\\
u\ar[rr]_{\widetilde{C}}\ar[d]_{\psi}\ar[rr]_{\widetilde{C}}\ar@/^{2.5pc}/[uu] &  & \widetilde{C}u\ar[d]^{\varphi}\ar@/_{2.5pc}/[uu]\\
\text{\ensuremath{\begin{bmatrix}u\\
-iB^{*}u
\end{bmatrix}}}\ar[rr]_{C}\ar@/^{2.5pc}/[uu] &  & \text{\ensuremath{\begin{bmatrix}\widetilde{C}u\\
iB^{*}\widetilde{C}u
\end{bmatrix}}}\ar@/_{2.5pc}/[uu]
}
\]

\caption{\label{fig:cc}The linear operator $\widetilde{C}$ in $N_{-1}\left(A^{*}B^{*}\right)$
induced by $C:N_{i}\left(L^{*}\right)\rightarrow N_{-i}\left(L^{*}\right)$. }

\end{figure}

\begin{rem}
Let $C:N_{i}\left(L^{*}\right)\rightarrow N_{-i}\left(L^{*}\right)$
be a partial isometry w.r.t. the $\mathscr{K}$ norm, i.e., $\left\Vert \cdot\right\Vert _{\mathscr{K}}^{2}=\left\Vert \cdot\right\Vert _{1}^{2}+\left\Vert \cdot\right\Vert _{2}^{2}$.
We conclude that 
\begin{equation}
\left\Vert u\right\Vert _{1}^{2}+\left\Vert B^{*}u\right\Vert _{2}^{2}=\left\Vert Qu\right\Vert _{1}^{2}+\left\Vert B^{*}Qu\right\Vert _{1}^{2},\quad\forall u\in N_{-1}\left(A^{*}B^{*}\right),\label{eq:nL1}
\end{equation}
where $Q:=\widetilde{C}$. 

It may occur that $A$ and $B$ are not closed; if not, refer to the
corresponding closures. Recall that $\overline{A}^{*}=A^{*}$, $\overline{B}^{*}=B^{*}$.
Then (\ref{eq:nL1}) takes the equivalent form
\begin{equation}
I_{1}+BB^{*}=Q^{*}Q+Q^{*}BB^{*}Q\label{eq:nL2}
\end{equation}
as an operator identity in $N_{-1}\left(A^{*}B^{*}\right)$. Equivalently
(the norm preserving property)
\begin{equation}
I_{1}+BB^{*}=Q^{*}\left(I+BB^{*}\right)Q,\label{eq:nL3}
\end{equation}
and so this is the property of $Q$ which is equivalent to the partial
isometric property of $C$. 
\end{rem}
\begin{cor}
Fix $L=\begin{bmatrix}0 & B\\
A & 0
\end{bmatrix}$, then the selfadjoint extensions $L_{Q}$ of $L$ are determined
by all operator solutions $Q$ to (\ref{eq:nL3}). 
\end{cor}
Moreover, 
\begin{equation}
dom\left(L_{Q}\right)=\left\{ \begin{bmatrix}x\\
y
\end{bmatrix}+\begin{bmatrix}u\\
-iB^{*}u
\end{bmatrix}+\begin{bmatrix}v\\
-iB^{*}v
\end{bmatrix}\right\} \label{eq:LQ1}
\end{equation}
where $\begin{bmatrix}x\\
y
\end{bmatrix}\in dom\left(L\right)$, $u,v\in N_{-1}\left(A^{*}B^{*}\right)$, and $v=Qu$; and 
\begin{equation}
L_{Q}\begin{bmatrix}x+u+v\\
y-iB^{*}u+iB^{*}v
\end{bmatrix}=\begin{bmatrix}By+iu-iv\\
Ax+B^{*}u-B^{*}v
\end{bmatrix}.\label{eq:LQ2}
\end{equation}

\begin{proof}
On the domain
\begin{equation}
dom\left(L_{C}\right)=\left\{ \varphi+\psi_{+}+C\psi_{+}\mid\varphi\in dom\left(L\right),\;\psi_{+}\in N_{i}\left(L^{*}\right)\right\} ,
\end{equation}
we have 
\begin{equation}
L_{C}\left(\varphi+\psi_{+}+C\psi_{+}\right)=L\varphi+i\psi_{+}-iC\psi_{+}.
\end{equation}
Now apply this (\ref{eq:LQ1})-(\ref{eq:LQ2}). Also see \cite{DS88b},
and \remref{vn}.
\end{proof}

\section{Selfadjoint extensions of semibounded operators}

Many \textquotedblleft naïve\textquotedblright{} treatments of linear
operators in the physics literature are based on analogies to finite
dimensions. They often result in paradoxes and inaccuracies as they
miss some key issues intrinsic to unbounded operators, questions dealing
with domains, closability, graphs, and in the symmetric case, the
distinction between formally Hermitian and selfadjoint, deficiency
indices, issues all inherent in infinite-dimensional analysis of unbounded
operators and their extensions. Only when these questions are resolved
for the particular application at hand, will we arrive at a rigorous
spectral analysis, and get reliable predictions of scattering (from
von Neumann's Spectral Theorem); see e.g., \cite{MR3167762,MR507913}.
Since measurements of the underlying observables, in prepared states,
come from the projection valued measures, which are dictated by choices
(i)-(ii) (see \secref{Intro}), these choices have direct physical
significance. 

Let $\mathscr{H}$ be a complex Hilbert space. Let $A$ be an operator
in $\mathscr{H}$ with $dom\left(A\right)=\mathscr{D}$, dense in
$\mathscr{H}$, such that 
\begin{equation}
\left\Vert \varphi\right\Vert _{A}^{2}:=\left\langle \varphi,A\varphi\right\rangle \geq\left\Vert \varphi\right\Vert ^{2},\quad\forall\varphi\in\mathscr{D}.\label{eq:s1}
\end{equation}
The completion of $\mathscr{D}$ with respect to the $\left\Vert \cdot\right\Vert _{A}$-norm
yields a Hilbert space $\mathscr{H}_{A}$. Let 
\[
J:\mathscr{H}_{A}\longrightarrow\mathscr{H},\quad J\varphi=\varphi,
\]
be the inclusion map. It follows from (\ref{eq:s1}) that
\begin{equation}
\left\Vert J\varphi\right\Vert =\left\Vert \varphi\right\Vert \leq\left\Vert \varphi\right\Vert _{A},\label{eq:s2}
\end{equation}
thus $J$ is contractive, and so are $J^{*}J$ and $JJ^{*}$. 
\begin{rem}
The inner product in $\mathscr{H}_{A}$ is denoted by $\left\langle \cdot,\cdot\right\rangle _{A}$
with subscript $A$, as opposed to $\left\langle \cdot,\cdot\right\rangle $
for the original Hilbert space $\mathscr{H}$. That is, 
\begin{equation}
\left\langle f,g\right\rangle _{A}:=\left\langle f,Ag\right\rangle ,\quad\forall f,g\in\mathscr{D}.\label{eq:s3}
\end{equation}
\end{rem}
Recall the adjoint operator $J^{*}:\mathscr{H}\longrightarrow\mathscr{H}_{A}$,
by 
\begin{equation}
\left\langle h,Jg\right\rangle =\left\langle J^{*}h,g\right\rangle _{A},\quad\forall h\in\mathscr{H},g\in\mathscr{H}_{A}.\label{eq:s4}
\end{equation}

\begin{thm}
\label{thm:Fe}The operator $\left(JJ^{*}\right)^{-1}$ is unbounded,
and is a selfadjoint extension of $A$, i.e., 
\begin{equation}
\left(JJ^{*}\right)^{-1}\supseteq A.\label{eq:s5}
\end{equation}
Moreover, it coincides with the Friedrichs extension \cite{DS88b}.
(See the diagram below.) 
\[
\xymatrix{\mathscr{H}_{A}\ar@/^{1pc}/[rr]^{J}\ar@(ul,dl)[]_{J^{*}J} &  & \mathscr{H}\ar@/^{1pc}/[ll]^{J^{*}}\ar@(dr,ur)[]_{JJ^{*}}}
\]
\end{thm}
\begin{proof}
(\ref{eq:s5})$\Longleftrightarrow$
\begin{alignat}{2}
\left(JJ^{*}\right)^{-1} & \varphi=A\varphi, &  & \forall\varphi\in\mathscr{D},\nonumber \\
 & \Updownarrow\nonumber \\
\varphi= & JJ^{*}A\varphi, &  & \forall\varphi\in\mathscr{D},\nonumber \\
 & \Updownarrow\nonumber \\
\left\langle \psi,\varphi\right\rangle = & \left\langle \psi,JJ^{*}A\varphi\right\rangle , & \quad & \forall\psi,\varphi\in\mathscr{D}.\label{eq:s6}
\end{alignat}
For a pair $\psi,\varphi\in\mathscr{D}$ as in (\ref{eq:s6}), we
have
\begin{alignat*}{2}
\mbox{RHS}_{\left(\ref{eq:s6}\right)} & =\left\langle J^{*}\psi,J^{*}A\varphi\right\rangle _{A} &  & \mbox{by }\left(\ref{eq:s4}\right)\\
 & =\left\langle JJ^{*}\psi,A\varphi\right\rangle  &  & \mbox{by }\left(\ref{eq:s4}\right)\\
 & =\left\langle J^{*}\psi,\varphi\right\rangle _{A} &  & \mbox{by }\left(\ref{eq:s3}\right),\mbox{ and \ensuremath{J^{**}=J} from general theory}\\
 & =\left\langle \psi,J\varphi\right\rangle \\
 & =\left\langle \psi,\varphi\right\rangle =\mbox{LHS}_{\left(\ref{eq:s6}\right)}
\end{alignat*}

That $\left(JJ^{*}\right)^{-1}$ is selfadjoint follows from a general
theorem of von Neumann (\thmref{vN}). See, e.g., \cite{DS88b}. $\left(JJ^{*}\right)^{-1}$
is the \emph{Friedrichs extension} of $A$. 
\end{proof}
Let $q$ be a sesquilinear form on $\mathscr{Q}\subset\mathscr{H}$
(linear in the second variable) such that: 
\begin{enumerate}[label=(\roman{enumi}),ref=\roman{enumi}]
\item \label{enu:q1}$\mathscr{Q}$ is a dense subspace in $\mathscr{H}$.
\item \label{enu:q2}$q\left(\varphi,\varphi\right)\geq\left\Vert \varphi\right\Vert ^{2}$,
for all $\varphi\in\mathscr{Q}$. 
\item \label{enu:q3}$q$ is \emph{closed}, i.e., $\mathscr{Q}$ is a Hilbert
space w.r.t. 
\begin{align*}
\left\langle \varphi,\psi\right\rangle _{q} & :=q\left(\varphi,\psi\right)\mbox{, and }\\
\left\Vert \varphi\right\Vert _{q}^{2} & :=q\left(\varphi,\varphi\right),\;\forall\varphi,\psi\in\mathscr{Q}.
\end{align*}
\end{enumerate}
\begin{cor}
There is a bijection between sesquilinear forms $q$ on $\mathscr{Q}\subset\mathscr{H}$
satisfying (\ref{enu:q1})-(\ref{enu:q3}), and selfadjoint operators
$A$ in $\mathscr{H}$ s.t. $A\geq1$. Specifically, the correspondence
is as follows:

\begin{enumerate}
\item \label{enu:a1}Given $A$, set $\mathscr{Q}:=dom(A^{\frac{1}{2}})$,
and 
\begin{equation}
q\left(\varphi,\psi\right):=\langle A^{\frac{1}{2}}\varphi,A^{\frac{1}{2}}\psi\rangle,\quad\forall\varphi,\psi\in dom(A^{\frac{1}{2}}).\label{eq:q}
\end{equation}
\item \label{enu:a2}Conversely, if $q$ satisfies (i)-(iii), let $J:\mathscr{Q}\rightarrow\mathscr{H}$
be the inclusion map, and set $A:=(JJ^{*})^{-1}$; then $q$ is determined
by the RHS of (\ref{eq:q}).
\end{enumerate}
\end{cor}
\begin{proof}
The non-trivial part (\ref{enu:a2}) $\Rightarrow$ (\ref{enu:a1})
follows from the proof of \thmref{Fe}.
\end{proof}
\begin{lem}
\label{lem:essA}Let $A$ be a semibounded operator as in (\ref{eq:s1}),
then $A$ is essentially selfadjoint iff $A\mathscr{D}$ is dense
in $\mathscr{H}$, i.e., $\overline{\mbox{ran}(A)}=\mathscr{H}$.
(Contrast, $\overline{A}=A^{**}$ denotes the closure of $A$.)
\end{lem}
\begin{proof}
Follows from von Neumann's deficiency index theory, and the assumption
that $A\geq1$ (see (\ref{eq:s1}).)
\end{proof}
By \lemref{essA}, if $A$ is not essentially selfadjoint, then 
\begin{equation}
C:A\varphi\longrightarrow\varphi\label{eq:s7}
\end{equation}
is contractive in $ran\left(A\right)$ (proper subspace in $\mathscr{H}$,
i.e., not dense in $\mathscr{H}$.)

Proof that (\ref{eq:s7}) is contractive: By (\ref{eq:s1}), we have
\[
\left\Vert \varphi\right\Vert ^{2}\underset{\left(\text{Schwarz}\right)}{\leq}\left\langle \varphi,A\varphi\right\rangle \leq\left\Vert \varphi\right\Vert \left\Vert A\varphi\right\Vert 
\]
which implies $\left\Vert \varphi\right\Vert \leq\left\Vert A\varphi\right\Vert $,
for all $\varphi\in\mathscr{D}$. 

We have proved that $CA\varphi=\varphi$ holds, and $C$ is s.a. and
contractive.
\begin{thm}[Krein \cite{MR0048704,MR0048703,MR0405049}]
 We introduce the set 
\begin{align}
\mathscr{B}_{A}:= & \big\{ B\mid B^{*}=B,\:dom\left(B\right)=\mathscr{H},\ \left\Vert Bh\right\Vert \leq\left\Vert h\right\Vert ,\;\forall h\in\mathscr{H},\label{eq:s8}\\
 & \quad\mbox{and }C\subset B\;\mbox{i.e., }CA\varphi=BA\varphi,\;\forall\varphi\in\mathscr{D};\;\mbox{see }\left(\ref{eq:s7}\right)\big\},\nonumber 
\end{align}
then $\mathscr{B}_{A}\neq\emptyset$. 
\end{thm}
\begin{cor}
For all $B\in\mathscr{B}_{A}$, we have $A\subset B^{-1}$ so $B^{-1}$
is an unbounded selfadjoint extension of $A$. 
\end{cor}
\begin{rem}
Krein studied $\mathscr{B}_{A}$ as an order lattice. Define $B_{1}\leq B_{2}$
meaning $\left\langle h,B_{1}h\right\rangle \leq\left\langle h,B_{2}h\right\rangle $,
$\forall h\in\mathscr{H}$. In the previous discussions we proved
that $JJ^{*}\in\mathscr{B}_{A}$.
\end{rem}

\section{Application to graph Laplacians, infinite networks}

We now turn to a family of semibounded operators from mathematical
physics. They arose first in the study of large (infinite) networks;
and in these studies entail important choices of Hilbert spaces, and
of selfadjoint realizations. The best known instance is perhaps systems
of resistors on infinite graphs, see e.g., \cite{MR920811,MR2735315,MR2862151,MR3096586,MR3246982,MR3441734}.
An early paper is \cite{MR0436847} which uses an harmonic analysis
of infinite systems of resistors in dealing with spin correlations
of states of finite energy of the isotropic ferromagnetic Heisenberg
model.

For the discussion of the graph Laplacian $\Delta$, we first introduce
the following setting of infinite networks:
\begin{itemize}
\item $V$: the vertex set, a given infinite countable discrete set.
\item $E\subset V\times V\backslash\left\{ \mbox{diagonal}\right\} $ the
edges, such that $\left(xy\right)\in E\Longleftrightarrow\left(yx\right)\in E$,
and for all $x\in V$, $\#\left\{ y\sim x\right\} <\infty$, where
$x\sim y$ means $\left(xy\right)\in E$.
\item $c:E\rightarrow\mathbb{R}_{+}$ a given conductance function. 
\item Set 
\begin{equation}
\left(\Delta u\right)\left(x\right):=\sum_{y\sim x}c_{xy}\left(u\left(x\right)-u\left(y\right)\right),\label{eq:s9}
\end{equation}
defined for all functions $u$ on $V$, and let 
\begin{equation}
c\left(x\right)=\sum_{y\sim x}c_{xy},\quad x\in V.\label{eq:tc}
\end{equation}
\item $\mathscr{H}_{E}$ will be the Hilbert space of finite-energy functions
on $V$; more precisely, 
\begin{equation}
u\in\mathscr{H}_{E}\underset{\mbox{Def.}}{\Longleftrightarrow}\left\Vert u\right\Vert _{\mathscr{H}_{E}}^{2}=\frac{1}{2}\sum_{\left(xy\right)\in E}c_{xy}\left|u\left(x\right)-u\left(y\right)\right|^{2}<\infty.\label{eq:he}
\end{equation}
Set 
\begin{equation}
\left\langle u,v\right\rangle _{\mathscr{H}_{E}}=\frac{1}{2}\sum_{\left(xy\right)\in E}c_{xy}(\overline{u\left(x\right)}-\overline{u\left(y\right)})\left(v\left(x\right)-v\left(y\right)\right).\label{eq:s10}
\end{equation}
\item We assume that $\left(V,E,c\right)$ is connected: For all pairs $x,y\in V$,
$\exists$ $\left(x_{i}\right)_{i=0}^{n}\subset V$ s.t. $x_{0}=x$,
$\left(x_{i}x_{i+1}\right)\in E$, $x_{n}=y$. 
\end{itemize}
\begin{lem}
\label{lem:dipole}Fix a base-point $o\in V$. Then for all $x\in V$,
there is a unique $v_{x}\in\mathscr{H}_{E}$ such that 
\begin{equation}
f\left(x\right)-f\left(o\right)=\left\langle v_{x},f\right\rangle _{\mathscr{H}_{E}},\quad\forall f\in\mathscr{H}_{E};\label{eq:g4}
\end{equation}
The vertex $v_{x}$ is called a \uline{dipole}. 
\end{lem}
\begin{proof}
see \cite{MR3096586,JT14c}. 
\end{proof}
\begin{lem}
\label{lem:del}In $\mathscr{H}_{E}$, we have $\delta_{x}=c\left(x\right)v_{x}-\sum_{y\sim x}c_{xy}v_{y}$,
and 
\[
\left|\left\langle \varphi,v_{x}\right\rangle _{\mathscr{H}_{E}}\right|=\left|\varphi\left(x\right)-\varphi\left(o\right)\right|\leq\sqrt{2}\left\Vert \varphi\right\Vert _{l^{2}},\quad\forall\varphi\in\mathscr{D}.
\]
\end{lem}
\begin{proof}
See \cite{MR2735315}.
\end{proof}
\begin{rem}
Let $\mathscr{H}=l^{2}\left(V\right)$, $\mathscr{D}=span\left\{ \delta_{x}\mid x\in V\right\} $.
Define the \emph{graph} \emph{Laplacian} $\Delta$ by (\ref{eq:s9}).
Let $\mathscr{H}_{E}$ be the energy-Hilbert space in (\ref{eq:he}).
Then (\ref{eq:s1}), (\ref{eq:s3}) translate into:
\begin{align}
\left\langle \delta_{x},\Delta\delta_{x}\right\rangle _{2} & =c\left(x\right)=\left\Vert \delta_{x}\right\Vert _{\mathscr{H}_{E}}^{2}\mbox{, and}\\
\left\langle \delta_{x},\Delta\delta_{y}\right\rangle _{2} & =-c_{xy}=\left\langle \delta_{x},\delta_{y}\right\rangle _{\mathscr{H}_{E}},\;\forall\left(xy\right)\in E,\;x\neq y.
\end{align}
Let $\mathscr{H}_{\Delta}$ be the completion of $\mathscr{D}=span\left\{ \delta_{x}\right\} $
with respect to $\left\langle \varphi,\Delta\varphi\right\rangle _{l^{2}}$,
$\varphi\in\mathscr{D}$. (We have $\left\langle \varphi,\Delta\varphi\right\rangle _{l^{2}}=\left\Vert \varphi\right\Vert _{\mathscr{H}_{E}}^{2}$,
valid for $\forall\varphi\in\mathscr{D}$.)

\textbf{Conclusion.} $\mathscr{H}_{\Delta}\hookrightarrow\mathscr{H}_{E}$
is an isometric inclusion, but as a subspace. The closure is $F_{in}=\mathscr{H}_{E}\ominus Harm$,
where $Harm$ is the subspace of Harmonic functions $h\in\mathscr{H}_{E}$,
i.e., $\Delta h=0$. 
\end{rem}
\begin{defn}
\emph{Two unbounded closable operators}:

The graph Laplacian is denoted by $\Delta_{2}$, as an operator in
$l^{2}$; and by $\Delta_{E}$ when acting in $\mathscr{H}_{E}$.
In both cases, $\Delta$ is given by (\ref{eq:s9}), defined for all
functions $u$ on $V$ . 
\end{defn}
\begin{defn}
Let $\left(V,E,c\right)$ be as before. Fix a base-point $o\in V$,
and let $v_{x}=v_{xo}=$ dipole (see \lemref{dipole}). Let 
\begin{align}
\mathscr{D}_{2} & =span\left\{ \delta_{x}\right\} \subset l^{2}\label{eq:d1}\\
\mathscr{D}_{E} & =span\left\{ v_{x}\right\} _{x\in V\backslash\left\{ o\right\} }\subset\mathscr{H}_{E}.\label{eq:d2}
\end{align}
Set 
\begin{alignat}{2}
l^{2} & \supset\mathscr{D}_{2}\xrightarrow{\;K\;}\mathscr{H}_{E}, & \quad & K(\delta_{x})=\delta_{x},\label{eq:s16}\\
\mathscr{H}_{E} & \supset\mathscr{D}_{E}\xrightarrow{\;L\;}l^{2}, & \quad & L(v_{x})=\delta_{x}-\delta_{o}.\label{eq:s17}
\end{alignat}
\end{defn}
\begin{lem}
We have 
\begin{equation}
\left\langle K\varphi,h\right\rangle _{\mathscr{H}_{E}}=\left\langle \varphi,Lh\right\rangle _{l^{2}},\quad\forall\varphi\in\mathscr{D}_{2},\forall h\in\mathscr{D}_{E}.\label{eq:KL1}
\end{equation}
\begin{figure}[H]
\[
\xymatrix{\mathscr{D}_{2}\subset l^{2}\ar@/^{2pc}/[rr]^{K}\ar@/^{1pc}/[rr]^{L^{*}} &  & \mathscr{H}_{E}\supset\mathscr{D}_{E}\ar@/^{2pc}/[ll]^{K^{*}}\ar@/^{1pc}/[ll]^{L}}
\]

\caption{\label{fig:KL}$dom\left(K\right)=\mathscr{D}_{2}$, $dom\left(L\right)=\mathscr{D}_{E}$,
$K\subset L^{*}$, and $L\subset K^{*}$.}
\end{figure}
\end{lem}
\begin{proof}
Note $K:l^{2}\rightarrow\mathscr{H}_{E}$ has dense domain $\mathscr{D}_{2}$
in $l^{2}$; and $J:\mathscr{H}_{E}\rightarrow l^{2}$ has dense domain
in $\mathscr{H}_{E}$. Moreover, it follows from (\ref{eq:KL1}) that 

(i) $K\subset L^{*}$, hence $dom(L^{*})$ is dense in $l^{2}$; and 

(ii) $L\subset K^{*}$, so $dom(K^{*})$ is dense in $\mathscr{H}_{E}$.
Also, both $K$ and $L$ are closable. See \figref{KL}. 

\emph{Proof of (\ref{eq:KL1})}: Use (\ref{eq:s9}) and linearity
to see that it is enough to consider the special case when $\varphi=\delta_{x}$,
$h=v_{y}$, so we must prove that the following holds ($x,y\in V$):
\begin{equation}
\left\langle K\delta_{x},v_{y}\right\rangle _{\mathscr{H}_{E}}=\left\langle \delta_{x},Lv_{y}\right\rangle _{2}.\label{eq:s14}
\end{equation}
Note that 
\begin{alignat*}{2}
\mbox{LHS}_{\left(\ref{eq:s14}\right)} & =\left\langle \delta_{x},v_{y}\right\rangle _{\mathscr{H}_{E}} &  & \text{by \ensuremath{\left(\ref{eq:s16}\right)}}\\
 & =\delta_{x}\left(y\right)-\delta_{x}\left(o\right) & \quad & \mbox{using the dipole property of \ensuremath{v_{y}}}\\
 & =\delta_{xy}-\delta_{xo};\\
\mbox{RHS}_{\left(\ref{eq:s14}\right)} & =\left\langle \delta_{x},\delta_{y}-\delta_{o}\right\rangle _{2} &  & \text{by \ensuremath{\left(\ref{eq:s17}\right)}}\\
 & =\delta_{xy}-\delta_{xo}.
\end{alignat*}
Thus (\ref{eq:s14}) holds. 
\end{proof}
\begin{cor}
\label{cor:KL}The two operators below are well-defined, and selfadjoint:
\begin{align}
 & \mbox{\ensuremath{K^{*}\overline{K}} is s.a. in \ensuremath{l^{2}}, and}\label{eq:st15}\\
 & \mbox{\ensuremath{L^{*}\overline{L}} is s.a. in \ensuremath{\mathscr{H}_{E}},}\label{eq:st16}
\end{align}
\textup{and both with dense domains}. Here, $\overline{\,\cdot\,}$
refers to the respective graph closures, and $*$ to adjoint operators,
i.e., $K^{*}:\mathscr{H}_{E}\longrightarrow l^{2}$, and $L^{*}:l^{2}\longrightarrow\mathscr{H}_{E}$;
both operators with dense domains, by (\ref{eq:KL1}). 

Moreover, (\ref{eq:st15})-(\ref{eq:st16}) are selfadjoint extensions
\begin{equation}
\Delta_{2}\subset K^{*}\overline{K}\mbox{ in \ensuremath{l^{2}}},\quad\mbox{and}\quad\Delta_{E}\subset L^{*}\overline{L}\:\mbox{in \ensuremath{\ensuremath{\mathscr{H}_{E}}}}.\label{eq:st17}
\end{equation}
In fact, $\overline{\Delta}_{2}=K^{*}\overline{K}$ (non-trivial;
see \cite{MR2432048,MR2799579}.)
\end{cor}
\begin{proof}
Conclusions (\ref{eq:st15})-(\ref{eq:st16}) follow from general
theory; see \thmref{vN}. To show 
\begin{equation}
\Delta_{E}\subset L^{*}\overline{L}\label{eq:KL2}
\end{equation}
we must prove that 
\begin{equation}
L^{*}\overline{L}v_{x}=\delta_{x}-\delta_{o}\left(=\Delta_{E}v_{x}\right),\quad\forall x\in V\backslash\left\{ o\right\} .\label{eq:st19}
\end{equation}
We have more: $\overline{K}=L^{*}$, and $\overline{L}=K^{*}$, but
this is because we have that $\Delta_{2}$ is essentially selfadjoint. 

To establish (\ref{eq:st19}), we must prove that the following equation
holds: 
\begin{equation}
\left\langle v_{y},L^{*}Lv_{x}\right\rangle _{\mathscr{H}_{E}}=\left\langle v_{y},\delta_{x}-\delta_{o}\right\rangle _{\mathscr{H}_{E}},\quad y\neq o.\label{eq:KL4}
\end{equation}
Note that 
\begin{align*}
\mbox{LHS}_{\left(\ref{eq:KL4}\right)} & =\left\langle Lv_{y},Lv_{x}\right\rangle _{2}\\
 & =\left\langle \delta_{y}-\delta_{o},\delta_{x}-\delta_{o}\right\rangle _{2}\quad\left(\mbox{by \ensuremath{\left(\ref{eq:s17}\right)}}\right)\\
 & =\delta_{xy}-\delta_{xo}-\delta_{yo}+\delta_{oo}=\delta_{xy}+1,\\
\mbox{RHS}_{\left(\ref{eq:KL4}\right)} & =\left(\delta_{x}-\delta_{o}\right)\left(y\right)-\left(\delta_{x}-\delta_{o}\right)\left(o\right)=\delta_{xy}+1.
\end{align*}

Now, using $\overline{J}=K^{*}$, we can show that $Harm\subset dom\left(L^{*}\overline{L}\right)=dom\left(L^{*}K^{*}\right)$,
and $L^{*}\overline{L}h=0$, which means that $L^{*}\overline{L}$
is the Krein extension of $\Delta_{E}$. 
\end{proof}
\vspace{1em}

\subsection*{Application of \thmref{H12}.}

Set $\mathscr{H}_{1}=l^{2}\left(V\right)$, $\mathscr{H}_{2}=\mathscr{H}_{E}$,
and let 
\begin{align*}
\mathscr{D} & :=\mathscr{D}_{2}=span\left\{ \delta_{x}\right\} _{x\in V},\;\text{and}\\
\mathscr{D}^{*} & :=\mathscr{D}_{E}=span\left\{ v_{x}\right\} _{x\in V\backslash\left\{ o\right\} };
\end{align*}
see (\ref{eq:d1}) \& (\ref{eq:d2}). Then the axioms (\ref{enu:t1})$\Longleftrightarrow$(\ref{enu:t2})
in \thmref{H12} hold. Note the only non-trivial part is the dense
subspace $\mathscr{D}^{*}\subset\mathscr{H}_{2}\left(=\mathscr{H}_{E}\right)$. 
\begin{claim}
The condition in (\ref{eq:t2}) holds; i.e., for all $h=v_{x}\in\mathscr{D}_{E}$,
there exists $C_{x}<\infty$ s.t. 
\begin{equation}
\left|\left\langle \varphi,v_{x}\right\rangle _{\mathscr{H}_{E}}\right|\leq C_{x}\left\Vert \varphi\right\Vert _{l^{2}},\quad\forall\varphi\in\mathscr{D}.\label{eq:d3}
\end{equation}
\end{claim}
\begin{proof}[Proof of (\ref{eq:d3})]
We have 
\begin{alignat*}{2}
\mbox{LHS}_{\left(\ref{eq:d3}\right)} & =\left|\left\langle \varphi,v_{x}\right\rangle _{\mathscr{H}_{E}}\right|\\
 & =\left|\varphi\left(x\right)-\varphi\left(o\right)\right| & \qquad & \mbox{by \ensuremath{\left(\ref{eq:g4}\right)}}\\
 & =\left|\left\langle \varphi,\delta_{x}-\delta_{o}\right\rangle _{l^{2}}\right|\\
 & \leq\left\Vert \varphi\right\Vert _{l^{2}}\left\Vert \delta_{x}-\delta_{o}\right\Vert _{l^{2}} & \qquad & \mbox{by Schwarz' inequality}\\
 & =\sqrt{2}\left\Vert \varphi\right\Vert _{l^{2}},\;\forall\varphi\in\mathscr{D},
\end{alignat*}
and so we may take $C_{x}=\sqrt{2}$. 
\end{proof}
\begin{rem}
For the setting in \thmref{H12} with $\mathscr{D}\subset\mathscr{H}_{1}\cap\mathscr{H}_{2}$,
note that the respective norms $\left\Vert \cdot\right\Vert _{i}$
on $\mathscr{H}_{i}$, $i=1,2$, induce norms $\left\Vert \cdot\right\Vert _{i}$
on $\mathscr{D}$. It is important that the conclusion in \thmref{H12}
is valid even when the two norms are not comparable; i.e., in general
there are no finite constants $C$, $D$ $\left(<\infty\right)$ such
that 
\begin{align}
\left\Vert \varphi\right\Vert _{1} & \leq C\left\Vert \varphi\right\Vert _{2},\quad\forall\varphi\in\mathscr{D};\mbox{ or}\label{eq:nc1}\\
\left\Vert \varphi\right\Vert _{2} & \leq D\left\Vert \varphi\right\Vert _{1},\quad\forall\varphi\in\mathscr{D}.\label{eq:nc2}
\end{align}
For the application above in \corref{KL}, the two Hilbert spaces
are:
\begin{itemize}
\item $\mathscr{H}_{1}=l^{2}\left(V\right)$
\item $\mathscr{H}_{2}=\mathscr{H}_{E}$ (the energy Hilbert space determined
from a fixed conductance function $c$), with $\mathscr{D}=span\left\{ \delta_{x}\mid x\in V\right\} $. 
\end{itemize}
Indeed, let $x\mapsto c\left(x\right)$ be the total conductance;
see (\ref{eq:tc}), then 
\[
\left\Vert \delta_{x}\right\Vert _{\mathscr{H}_{E}}^{2}=c\left(x\right)\quad\text{and}\quad\left\Vert \delta_{x}\right\Vert _{l^{2}}^{2}=1,
\]
so (\ref{eq:nc2}) does not hold when $c\left(\cdot\right)$ is unbounded
on $V$. (To see this, take $\varphi=\delta_{x}$.) 

From the analysis above, and \cite{JT14c,MR2799579} there are many
examples  such that $spec_{l^{2}}\left(\Delta_{2}\right)=[0,\infty)$.
One checks that in these examples, the estimate (\ref{eq:nc1}) also
will not hold for any finite constant $C$, i.e., $\left\Vert \cdot\right\Vert _{1}=\left\Vert \cdot\right\Vert _{l^{2}}$,
and $\left\Vert \cdot\right\Vert _{2}=\left\Vert \cdot\right\Vert _{\mathscr{H}_{E}}$. 
\end{rem}

\subsection*{Application of \thmref{gp}}

We apply the general symmetric pair $\left(A,B\right)$ to $\left(V,E,c\right)$:
\[
\xymatrix{\mathscr{H}_{1}\ar@/^{1.2pc}/[r]^{A} & \mathscr{H}_{2}\ar@/^{1.2pc}/[l]^{B} &  & l^{2}\left(V\right)\ar@/^{1.2pc}/[r]^{A}\ar@<-0.5ex>[r]_{B^{*}} & \mathscr{H}_{E}\ar@/^{1.2pc}/[l]^{B}\ar@<-0.5ex>[l]_{A^{*}}}
\]
Notation:
\begin{itemize}
\item $\mathscr{D}=span\left\{ \delta_{x}\mid x\in V\backslash\left\{ o\right\} \right\} =$
finitely supported functions on $V\backslash\left\{ o\right\} $
\item $l^{2}:=l^{2}\left(V\backslash\left\{ o\right\} \right)$
\item $\mathscr{H}_{E}=$ the corresponding energy Hilbert space
\item $\mathscr{K}=l^{2}\oplus\mathscr{H}_{E}\left(=\mathscr{H}_{1}\oplus\mathscr{H}_{2}\right)$
\end{itemize}
The pair $\left(A,B\right)$ is maximal, where $A$ and $B$ are defined
as follows:
\begin{align}
l^{2} & \ni\delta_{x}\xrightarrow{\;A\;}\delta_{x}=c\left(x\right)v_{x}-\sum_{y\sim x}c_{xy}v_{y}\in\mathscr{H}_{E}\quad\left(\text{Lemma. }\ref{lem:del}\right);\\
\mathscr{H}_{E} & \ni v_{x}\xrightarrow{\;B\;}\delta_{x}-\delta_{o}\in l^{2},\;\text{i.e., }B=\Delta.
\end{align}
Then $\mathscr{D}\subset l^{2}\cap\mathscr{H}_{E}$, and both inclusions
are isometric. 

Define $L=\begin{bmatrix}0 & B\\
A & 0
\end{bmatrix}$ on $\mathscr{K}=l^{2}\oplus\mathscr{H}_{E}$, where
\begin{align}
dom\left(L\right) & :=\left\{ \begin{bmatrix}\varphi\\
f
\end{bmatrix}\mid\varphi\in\mathscr{D},\;f\in dom\left(\Delta\right)\right\} ,\;\text{and}\label{eq:gL1}\\
L\begin{bmatrix}\varphi\\
f
\end{bmatrix} & :=\begin{bmatrix}Bf\\
A\varphi
\end{bmatrix}=\begin{bmatrix}\Delta f\\
\varphi
\end{bmatrix},\quad\forall\begin{bmatrix}\varphi\\
f
\end{bmatrix}\in dom\left(L\right).\label{eq:gL2}
\end{align}

It follows that $L$ is a Hermitian symmetric operator in $\mathscr{K}$,
i.e., $L\subseteq L^{*}$, but we must have:
\begin{thm}
\label{thm:gL}The operator $L$ in (\ref{eq:gL2}) is essentially
selfadjoint in the Hilbert space $\mathscr{K}$, i.e., it has deficiency
indices $\left(d_{+},d_{-}\right)=\left(0,0\right)$. 
\end{thm}
\begin{proof}
\textbf{Step 1.} We have 
\begin{equation}
\left\langle A\varphi,f\right\rangle _{\mathscr{H}_{E}}=\left\langle \varphi,Bf\right\rangle _{l^{2}},\quad\forall\varphi,f\in\mathscr{D},\label{eq:aL1}
\end{equation}
so that $A\subseteq B^{*}$ and $B\subseteq A^{*}$. 

\textbf{Step 2.} Define $L$ as in (\ref{eq:gL1})-(\ref{eq:gL2}).
For the adjoint operator, set $L^{*}=\begin{bmatrix}0 & A^{*}\\
B^{*} & 0
\end{bmatrix}$, with
\begin{align}
dom\left(L^{*}\right) & =\begin{bmatrix}dom\left(B^{*}\right)\\
dom\left(A^{*}\right)
\end{bmatrix},\;\text{and}\\
L^{*}\begin{bmatrix}h_{1}\\
h_{2}
\end{bmatrix} & =\begin{bmatrix}A^{*}h_{2}\\
B^{*}h_{1}
\end{bmatrix},\quad h_{1}\in l^{2},\:h_{2}\in\mathscr{H}_{E}.
\end{align}
So we must be precise about $A^{*}$ and $B^{*}$, and we shall need
the following:
\end{proof}
\begin{lem}
\label{lem:gLad}The domains of $A^{*}$ and $B^{*}$are as follows:
\begin{align}
dom\left(A^{*}\right)=\Big\{ & f\in\mathscr{H}_{E}\mid\exists C_{f}<\infty\;s.t.\;\nonumber \\
 & \left|\left\langle \varphi,f\right\rangle _{\mathscr{H}_{E}}\right|^{2}\leq C_{f}\left\Vert \varphi\right\Vert _{2}^{2}=C_{f}\sum_{x}\left|\varphi_{x}\right|^{2}\Big\};
\end{align}
and 
\begin{align}
dom\left(B^{*}\right)= & \Big\{\varphi\in l^{2}\mid\exists C_{\varphi}<\infty\;s.t.\;\nonumber \\
 & \left|\left\langle \varphi,\Delta f\right\rangle _{l^{2}}\right|^{2}\leq C_{\varphi}\left\Vert f\right\Vert _{\mathscr{H}_{E}}^{2},\;\forall f\in\mathscr{H}_{E}\;s.t.\;\Delta f\in l^{2}\Big\}.
\end{align}
\end{lem}
\begin{proof}
See the definitions and (\ref{eq:aL1}).
\end{proof}
\begin{rem}
\label{rem:glD}It is convenient to use $\Delta$ to act on all functions,
and later to adjoint domains. See the definition in (\ref{eq:s9}),
i.e., 
\begin{equation}
\left(\Delta u\right)\left(x\right):=\sum_{y\sim x}c_{xy}\left(u\left(x\right)-u\left(y\right)\right),\quad f\in\mathscr{F}\left(V\right)\left(=\text{all functions}\right).\label{eq:gLL}
\end{equation}
\end{rem}
\begin{proof}[Proof of \thmref{gL} continued]~

\textbf{Step 3.} Recall that $dom\left(A^{*}\right)\subset\mathscr{H}_{E}$,
and $dom\left(B^{*}\right)\subset l^{2}$: $\xymatrix{l^{2}\ar@<-0.5ex>[r]_{B^{*}} & \mathscr{H}_{E}\ar@<-0.5ex>[l]_{A^{*}}}
$. It follows from \lemref{gLad}, that 
\begin{alignat}{2}
\left(A^{*}f\right)\left(x\right) & =\left(\Delta f\right)\left(x\right), & \quad & \forall f\in dom\left(A^{*}\right),\;x\in V,\;\text{and}\\
B^{*}\underset{\text{in }l^{2}}{\underbrace{\varphi}} & =\varphi\in\mathscr{H}_{E}, &  & \forall\varphi\in dom\left(B^{*}\right).\label{eq:glB}
\end{alignat}
Note both sides of (\ref{eq:glB}) are interpreted as functions on
$V$ and the condition on $\varphi$ to be in $dom\left(B^{*}\right)$
is that $\underset{\left(xy\right)\in E}{\sum\sum}c_{cy}\left(\varphi\left(x\right)-\varphi\left(y\right)\right)^{2}<\infty$,
and also $\sum_{x}\varphi_{x}^{2}<\infty$. 

\textbf{Step 4.} Now consider $\Delta$ (in (\ref{eq:gLL}), see \remref{glD}),
then the two eigenvalue problems: 
\begin{equation}
\left\{ \begin{matrix}B^{*}A^{*}f=-f\\
A^{*}B^{*}\varphi=-\varphi
\end{matrix}\right\} \Longleftrightarrow\left\{ \begin{matrix}\Delta f=-f\\
\Delta\varphi=-\varphi
\end{matrix}\right\} \label{eq:glI}
\end{equation}
where $f\in\mathscr{H}_{E}$, $\Delta f\in l^{2}$, and $\varphi\in l^{2}\cap\mathscr{H}_{E}$. 

Apply the two isomorphisms from the general theory (see (\ref{eq:gp6})).
But (\ref{eq:glI}) only has the solution $\varphi=0$ in $l^{2}$.
The fact that (\ref{eq:glI}) does not have non-zero solutions follows
from \cite{MR2432048,MR2799579}. So we have $L=\begin{bmatrix}0 & B\\
A & 0
\end{bmatrix}$ \emph{essentially selfadjoint}. Indeed, this holds in the general
case. 

\textbf{Step 5.} The deficiency indices of the operator $L$. With
the definitions, 
\[
L=\begin{bmatrix}0 & B\\
A & 0
\end{bmatrix},\quad L\begin{bmatrix}\varphi\\
f
\end{bmatrix}=\begin{bmatrix}Bf\\
A\varphi
\end{bmatrix}=\begin{bmatrix}\Delta f\\
\varphi
\end{bmatrix}
\]
where $\varphi\in l^{2}$, $f\in\mathscr{H}_{E}$ are in the suitable
domains s.t. 
\begin{equation}
\left\Vert L\begin{bmatrix}\varphi\\
f
\end{bmatrix}\right\Vert _{\stackrel[\mathscr{H}_{E}]{l^{2}}{\oplus}}^{2}=\left\Vert \Delta f\right\Vert _{l^{2}}^{2}+\left\Vert \varphi\right\Vert _{\mathscr{H}_{E}}^{2}<\infty.\label{eq:gLn}
\end{equation}
So $\varphi\in l^{2}\cap\mathscr{H}_{E}$ , $f\in\mathscr{H}_{E}$,
$\Delta f\in l^{2}$ defines the domain of $L$ as an operator in
$\mathscr{K}=\left[\stackrel[\mathscr{H}_{E}]{l^{2}}{\oplus}\right]$,
and we proved that $L$ is selfadjoint, so indices $\left(0,0\right)$.

\end{proof}
\begin{cor}
Viewing $L$ as a selfadjoint operator, it follows from (\ref{eq:gLn})
that 
\[
dom\left(L\right)=\left\{ \begin{bmatrix}\varphi\\
f
\end{bmatrix}\in\left[\stackrel[\mathscr{H}_{E}]{l^{2}}{\oplus}\right]\mid\varphi\in l^{2}\cap\mathscr{H}_{E},\;\Delta f\in l^{2}\right\} .
\]

\bigskip{}
\end{cor}
\begin{acknowledgement*}
The co-authors thank the following colleagues for helpful and enlightening
discussions: Professors Daniel Alpay, Sergii Bezuglyi, Ilwoo Cho,
Paul Muhly, Myung-Sin Song, Wayne Polyzou, and members in the Math
Physics seminar at The University of Iowa.
\end{acknowledgement*}
\bibliographystyle{amsalpha}
\bibliography{ref}

\end{document}